\documentclass{amsart}

\usepackage{amsthm,amsfonts}
\usepackage{amssymb}
\usepackage{bbm} 
\usepackage{graphicx}
\usepackage[msc-links]{amsrefs}
\usepackage{color}
\usepackage{tikz-cd} 
\usepackage{mathrsfs}
\usepackage{dutchcal}

\DefineSimpleKey{bib}{primaryclass}{}
\DefineSimpleKey{bib}{archiveprefix}{}
\BibSpec{arXiv}{%
	+{}{\PrintAuthors}{author}
	+{,}{ \textit}{title}
	+{}{ \parenthesize}{date}
	+{,}{ arXiv }{eprint}
	+{,}{ primary class }{primaryclass}
}

\font\cyrillic=wncyi10
\newcommand{\comperm}{\hbox{\cyrillic X}\,}

\newcommand{\prop}[1]{{\mathcal{#1}}}
\newcommand{\cate}[1]{{\mathsf{#1}}}

\newcommand{\field}{{\mathbb K}}

\newcommand{\bS}{{\mathbb S}}
\newcommand{\bX}{{\mathbb X}}
\DeclareMathOperator{\Id}{\mathbb{I}}
\DeclareMathOperator{\b1}{\mathbbm{1}}
\DeclareMathOperator{\Orb}{Orb}

\DeclareMathOperator{\End}{End}

\DeclareMathOperator{\Sh}{Sh}
\newcommand{\comp}{\mathscr{C}}
\DeclareMathOperator{\Sol}{Sol}
\newcommand{\PT}{\mathcal{T}}
\DeclareMathOperator{\Des}{Des}	
\DeclareMathOperator{\Int}{\mathscr{I}}
 
\DeclareMathOperator{\Partitions}{\mathcal{Part}}

\DeclareMathOperator{\Prim}{Prim}

\DeclareMathOperator{\sol}{\prop{P}\text{-}\mathrm{Sol}}
\DeclareMathOperator{\magSol}{\prop{Mag}\text{-}\mathrm{Sol}}
\newcommand{\conj}[1]{[#1]}

\newcommand{\Vect}{\cate{Vect}}
\newcommand{\Smod}{\cate{\bS}\text{-}\cate{mod}}
\newcommand{\grVect}{\cate{gr}\text{-}\cate{Vect}}
\newcommand{\Und}{\cate{Und}}
\DeclareMathOperator{\propP}{\prop{P}}
\DeclareMathOperator{\propH}{\prop{H}}
\DeclareMathOperator{\Mag}{\prop{Mag}}
\DeclareMathOperator{\propLoop}{\prop{Loop}}
\DeclareMathOperator{\TMag}{\prop{3Mag}}

\DeclareMathOperator{\propCom}{\prop{C}\prop{o}\prop{m}}
\DeclareMathOperator{\propAssoc}{\prop{Assoc}}
\DeclareMathOperator{\propLie}{\prop{Lie}}
\DeclareMathOperator{\propSab}{\prop{Sab}}
\DeclareMathOperator{\propMag}{\prop{Mag}}
\DeclareMathOperator{\propC}{\prop{C}}

\DeclareMathOperator{\utimes}{\underline{\times}} 
\DeclareMathOperator{\uotimes}{\underline{\otimes}}
\DeclareMathOperator{\longi}{\mathit{l}} 
\DeclareMathOperator{\st}{st} 
\DeclareMathOperator{\ideal}{ideal}

\DeclareMathOperator{\Nil}{Nil}
\DeclareMathOperator{\spann}{\mathrm{span}} 
\DeclareMathOperator{\Stb}{Stb}
\DeclareMathOperator{\typ}{\mathrm{type}} 	
\DeclareMathOperator{\ctyp}{\mathrm{ctype}}

\newtheorem{theorem}{Theorem}[section]
\newtheorem{lemma}[theorem]{Lemma}
\newtheorem{proposition}[theorem]{Proposition}
\newtheorem{corollary}[theorem]{Corollary}

\theoremstyle{definition}
\newtheorem{definition}[theorem]{Definition}
\newtheorem{example}[theorem]{Example}

\theoremstyle{remark}
\newtheorem{remark}[theorem]{Remark}

\numberwithin{equation}{section}

\input{xy}
\xyoption{all}

\newcount\grcalca
\newcount\grcalcb
\newcount\grcalcc
\newcount\grcalcd
\newcount\grcalce
\newcount\grcalcf
\newcount\grcalcg
\newcount\grcalch
\newcount\grcalci
\newcount\grrow
\newcount\grcolumn
\newcount\grwidth
\newcount\factor
\newcount\hfactor
\newcount\qfactor
\newcount\tfactor
\newcount\sfactor
\newcount\dfactor
\hfactor = \factor
\divide    \hfactor by 2
\qfactor = \hfactor
\divide    \qfactor by 2
\tfactor = \factor
\divide    \tfactor by 3
\sfactor = \factor
\divide    \sfactor by 6
\dfactor = \factor
\divide    \dfactor by 12

\newcommand{\gbeg}[2]{
	\unitlength=1pt
	\grrow = #2
	\grcolumn = 0
	\grcalca = #1
	\grcalcb = #2
	\multiply \grcalca by \factor
	\grwidth = \grcalca
	\multiply \grcalcb by \factor
	\begin{minipage}{\grcalca pt}
		\begin{picture}(\grcalca,\grcalcb)
			\advance \grcalcb by -\factor
		}
		
		\newcommand{\gend}{
		\end{picture}
		{\vskip2.5ex}
\end{minipage} }

\newcommand{\gnl}{
	\advance \grrow by -1
	\grcolumn = 0}

\newcommand{\gvac}[1]{
	\advance \grcolumn by #1}

\newcommand{\gcl}[1]{
	\grcalca = \grcolumn
	\multiply \grcalca by \factor
	\advance \grcalca by \hfactor
	\grcalcb = \grrow
	\multiply \grcalcb by \factor
	\grcalcc = #1
	\multiply \grcalcc by \factor
	\put(\grcalca,\grcalcb) {\line(0,-1){\grcalcc}}
	\advance \grcolumn by 1}

\newcommand{\gcn}[4]{
	\grcalca = \grcolumn
	\multiply \grcalca by \factor
	\grcalci = #3
	\multiply \grcalci by \hfactor
	\advance \grcalca by \grcalci
	\grcalcb = \grcolumn
	\multiply \grcalcb by \factor
	\grcalci = #3
	\advance \grcalci by #4
	\multiply \grcalci by \qfactor
	\advance \grcalcb by \grcalci
	\grcalcc = \grcolumn
	\multiply \grcalcc by \factor
	\grcalci = #4
	\multiply \grcalci by \hfactor
	\advance \grcalcc by \grcalci
	\grcalcd = \grrow
	\multiply \grcalcd by \factor
	\grcalce = \grrow
	\multiply \grcalce by \factor
	\grcalci = #2
	\multiply \grcalci by \tfactor
	\advance \grcalce by -\grcalci
	\grcalcf = \grrow
	\multiply \grcalcf by \factor
	\grcalci = #2
	\multiply \grcalci by \hfactor
	\advance \grcalcf by -\grcalci
	\grcalcg = \grrow
	\multiply \grcalcg by \factor
	\grcalci = #2
	\multiply \grcalci by \tfactor
	\multiply \grcalci by 2
	\advance \grcalcg by -\grcalci
	\grcalch = \grrow
	\advance \grcalch by -#2
	\multiply \grcalch by \factor
	\qbezier(\grcalca,\grcalcd)(\grcalca,\grcalce)(\grcalcb,\grcalcf)
	\qbezier(\grcalcb,\grcalcf)(\grcalcc,\grcalcg)(\grcalcc,\grcalch)
	\advance \grcolumn by #1}

\newcommand{\gnot}[1]{
	\grcalca = \grcolumn
	\multiply \grcalca by \factor
	\advance \grcalca by \hfactor
	\grcalcb = \grrow
	\multiply \grcalcb by \factor
	\advance \grcalcb by -\hfactor
	\put(\grcalca,\grcalcb) {\makebox(0,0){$\scriptstyle #1$}} }

\newcommand{\got}[2]{
	\grcalca = \grcolumn
	\multiply \grcalca by \factor
	\grcalcc = #1
	\multiply \grcalcc by \hfactor
	\advance \grcalca by \grcalcc
	\grcalcb = \grrow
	\multiply \grcalcb by \factor
	\advance \grcalcb by -\tfactor
	\advance \grcalcb by -\tfactor
	\put(\grcalca,\grcalcb){\makebox(0,0)[b]{$#2$}}
	\advance \grcolumn by #1}

\newcommand{\gob}[2]{
	\grcalca = \grcolumn
	\multiply \grcalca by \factor
	\grcalcc = #1
	\multiply \grcalcc by \hfactor
	\advance \grcalca by \grcalcc
	\put(\grcalca,0){\makebox(0,0)[b]{$#2$}}
	\advance \grcolumn by #1}

\newcommand{\gmu}{
	\grcalca = \grcolumn
	\advance \grcalca by 1
	\multiply \grcalca by \factor
	\grcalcb = \grrow
	\multiply \grcalcb by \factor
	\grcalcc = \factor
	\advance \grcalcc by \hfactor
	\put(\grcalca,\grcalcb){\oval(\factor,\grcalcc)[b]}
	\advance \grcalcb by -\hfactor
	\advance \grcalcb by -\qfactor
	\put(\grcalca,\grcalcb) {\line(0,-1){\qfactor}}
	\advance \grcolumn by 2}

\newcommand{\gcmu}{
	\grcalca = \grcolumn
	\advance \grcalca by 1
	\multiply \grcalca by \factor
	\grcalcb = \grrow
	\advance \grcalcb by -1
	\multiply \grcalcb by \factor
	\grcalcc = \factor
	\advance \grcalcc by \hfactor
	\put(\grcalca,\grcalcb){\oval(\factor,\grcalcc)[t]}
	\advance \grcalcb by \factor
	\put(\grcalca,\grcalcb) {\line(0,-1){\qfactor}}
	\advance \grcolumn by 2}

\newcommand{\glm}{
	\grcalca = \grcolumn
	\multiply \grcalca by \factor
	\advance \grcalca by \hfactor
	\grcalcb = \grcalca
	\advance \grcalcb by \factor
	\grcalcc = \grrow
	\multiply \grcalcc by \factor
	\grcalcd = \grcalcc
	\advance \grcalcd by -\tfactor
	\grcalce = \grcalcd
	\advance \grcalce by -\tfactor
	\put(\grcalca, \grcalcc){\line(0,-1){\tfactor}}
	\put(\grcalca, \grcalcd){\line(1,0){\factor}}
	\put(\grcalca, \grcalcd){\line(3,-1){\factor}}
	\put(\grcalcb, \grcalcc){\line(0,-1){\factor}}
	\advance \grcolumn by 2}

\newcommand{\grm}{
	\grcalcb = \grcolumn
	\multiply \grcalcb by \factor
	\advance \grcalcb by \hfactor
	\grcalca = \grcalcb
	\advance \grcalca by \factor
	\grcalcc = \grrow
	\multiply \grcalcc by \factor
	\grcalcd = \grcalcc
	\advance \grcalcd by -\tfactor
	\grcalce = \grcalcd
	\advance \grcalce by -\tfactor
	\put(\grcalca, \grcalcc){\line(0,-1){\tfactor}}
	\put(\grcalca, \grcalcd){\line(-1,0){\factor}}
	\put(\grcalca, \grcalcd){\line(-3,-1){\factor}}
	\put(\grcalcb, \grcalcc){\line(0,-1){\factor}}
	\advance \grcolumn by 2}

\newcommand{\glcm}{
	\grcalca = \grcolumn
	\multiply \grcalca by \factor
	\advance \grcalca by \hfactor
	\grcalcb = \grcalca
	\advance \grcalcb by \factor
	\grcalcc = \grrow
	\advance \grcalcc by -1
	\multiply \grcalcc by \factor
	\grcalcd = \grcalcc
	\advance \grcalcd by \tfactor
	\grcalce = \grcalcd
	\advance \grcalce by \tfactor
	\put(\grcalca, \grcalcc){\line(0,1){\tfactor}}
	\put(\grcalca, \grcalcd){\line(1,0){\factor}}
	\put(\grcalca, \grcalcd){\line(3,1){\factor}}
	\put(\grcalcb, \grcalcc){\line(0,1){\factor}}
	\advance \grcolumn by 2}

\newcommand{\grcm}{
	\grcalcb = \grcolumn
	\multiply \grcalcb by \factor
	\advance \grcalcb by \hfactor
	\grcalca = \grcalcb
	\advance \grcalca by \factor
	\grcalcc = \grrow
	\advance \grcalcc by -1
	\multiply \grcalcc by \factor
	\grcalcd = \grcalcc
	\advance \grcalcd by \tfactor
	\grcalce = \grcalcd
	\advance \grcalce by \tfactor
	\put(\grcalca, \grcalcc){\line(0,1){\tfactor}}
	\put(\grcalca, \grcalcd){\line(-1,0){\factor}}
	\put(\grcalca, \grcalcd){\line(-3,1){\factor}}
	\put(\grcalcb, \grcalcc){\line(0,1){\factor}}
	\advance \grcolumn by 2}

\newcommand{\gwmu}[1]{
	\grcalca = \grcolumn
	\multiply \grcalca by \factor
	\grcalcd = \hfactor
	\multiply \grcalcd by #1
	\advance \grcalca by \grcalcd
	\grcalcb = \grrow
	\multiply \grcalcb by \factor
	\grcalcc = \factor
	\advance \grcalcc by \hfactor
	\grcalcd = #1
	\advance \grcalcd by -1
	\multiply \grcalcd by \factor
	\put(\grcalca,\grcalcb){\oval(\grcalcd,\grcalcc)[b]}
	\advance \grcalcb by -\hfactor
	\advance \grcalcb by -\qfactor
	\put(\grcalca,\grcalcb) {\line(0,-1){\qfactor}}
	\advance \grcolumn by #1}

\newcommand{\gwcm}[1]{
	\grcalca = \grcolumn
	\multiply \grcalca by \factor
	\grcalcd = \hfactor
	\multiply \grcalcd by #1
	\advance \grcalca by \grcalcd
	\grcalcb = \grrow
	\advance \grcalcb by -1
	\multiply \grcalcb by \factor
	\grcalcc = \factor
	\advance \grcalcc by \hfactor
	\grcalcd = #1
	\advance \grcalcd by -1
	\multiply \grcalcd by \factor
	\put(\grcalca,\grcalcb){\oval(\grcalcd,\grcalcc)[t]}
	\advance \grcalcb by \factor
	\put(\grcalca,\grcalcb) {\line(0,-1){\qfactor}}
	\advance \grcolumn by #1}

\newcommand{\gwmuc}[1]{
	\grcalca = \grcolumn
	\multiply \grcalca by \factor
	\advance \grcalca by \hfactor
	\grcalcb = \grrow
	\multiply \grcalcb by \factor
	\grcalcc = #1
	\advance \grcalcc by -1
	\multiply \grcalcc by \factor
	\put(\grcalca,\grcalcb){\line(1,0){\grcalcc}}
	\advance \grcalca by -\hfactor
	\grcalcd = \hfactor
	\multiply \grcalcd by #1
	\advance \grcalca by \grcalcd
	\grcalcc = \factor
	\advance \grcalcc by \hfactor
	\grcalcd = #1
	\advance \grcalcd by -1
	\multiply \grcalcd by \factor
	\put(\grcalca,\grcalcb){\oval(\grcalcd,\grcalcc)[b]}
	\advance \grcalcb by -\hfactor
	\advance \grcalcb by -\qfactor
	\put(\grcalca,\grcalcb) {\line(0,-1){\qfactor}}
	\advance \grcolumn by #1}

\newcommand{\gwcmc}[1]{
	\grcalca = \grcolumn
	\multiply \grcalca by \factor
	\advance \grcalca by \hfactor
	\grcalcb = \grrow
	\multiply \grcalcb by \factor
	\advance \grcalcb by -\factor
	\grcalcc = #1
	\advance \grcalcc by -1
	\multiply \grcalcc by \factor
	\put(\grcalca,\grcalcb){\line(1,0){\grcalcc}}
	\grcalcd = #1
	\advance \grcalcd by -1
	\multiply \grcalcd by \hfactor
	\advance \grcalca by \grcalcd
	\grcalcc = \factor
	\advance \grcalcc by \hfactor
	\grcalcd = #1
	\advance \grcalcd by -1
	\multiply \grcalcd by \factor
	\put(\grcalca,\grcalcb){\oval(\grcalcd,\grcalcc)[t]}
	\advance \grcalcb by \factor
	\put(\grcalca,\grcalcb) {\line(0,-1){\qfactor}}
	\advance \grcolumn by #1}

\newcommand{\gev}{
	\grcalca = \grcolumn
	\advance \grcalca by 1
	\multiply \grcalca by \factor
	\grcalcb = \grrow
	\multiply \grcalcb by \factor
	\grcalcc = \factor
	\advance \grcalcc by \hfactor
	\put(\grcalca,\grcalcb){\oval(\factor,\grcalcc)[b]}
	\advance \grcolumn by 2}

\newcommand{\gdb}{
	\grcalca = \grcolumn
	\advance \grcalca by 1
	\multiply \grcalca by \factor
	\grcalcb = \grrow
	\advance \grcalcb by -1
	\multiply \grcalcb by \factor
	\grcalcc = \factor
	\advance \grcalcc by \hfactor
	\put(\grcalca,\grcalcb){\oval(\factor,\grcalcc)[t]}
	\advance \grcolumn by 2}

\newcommand{\gwev}[1]{
	\grcalca = \grcolumn
	\multiply \grcalca by \factor
	\grcalcd = \hfactor
	\multiply \grcalcd by #1
	\advance \grcalca by \grcalcd
	\grcalcb = \grrow
	\multiply \grcalcb by \factor
	\grcalcc = \factor
	\advance \grcalcc by \hfactor
	\grcalcd = #1
	\advance \grcalcd by -1
	\multiply \grcalcd by \factor
	\put(\grcalca,\grcalcb){\oval(\grcalcd,\grcalcc)[b]}
	\advance \grcolumn by #1}

\newcommand{\gwdb}[1]{
	\grcalca = \grcolumn
	\multiply \grcalca by \factor
	\grcalcd = \hfactor
	\multiply \grcalcd by #1
	\advance \grcalca by \grcalcd
	\grcalcb = \grrow
	\advance \grcalcb by -1
	\multiply \grcalcb by \factor
	\grcalcc = \factor
	\advance \grcalcc by \hfactor
	\grcalcd = #1
	\advance \grcalcd by -1
	\multiply \grcalcd by \factor
	\put(\grcalca,\grcalcb){\oval(\grcalcd,\grcalcc)[t]}
	\advance \grcolumn by #1}

\newcommand{\gbr}{
	\grcalca = \grcolumn
	\multiply \grcalca by \factor
	\advance \grcalca by \hfactor
	\grcalcb = \grcalca
	\advance \grcalcb by \hfactor
	\grcalcc = \grcalca
	\advance \grcalcc by \factor
	\grcalcd = \grrow
	\multiply \grcalcd by \factor
	\grcalce = \grcalcd
	\advance \grcalce by -\tfactor
	\grcalcf = \grcalcd
	\advance \grcalcf by -\hfactor
	\grcalcg = \grcalce
	\advance \grcalcg by -\tfactor
	\grcalch = \grcalcd
	\advance \grcalch by -\factor
	\qbezier(\grcalca,\grcalcd)(\grcalca,\grcalce)(\grcalcb,\grcalcf)
	\qbezier(\grcalcb,\grcalcf)(\grcalcc,\grcalcg)(\grcalcc,\grcalch)
	\advance \grcalcf by -\dfactor
	\advance \grcalcb by -\sfactor
	\qbezier(\grcalca,\grcalch)(\grcalca,\grcalcg)(\grcalcb,\grcalcf)
	\advance \grcalcf by \sfactor
	\advance \grcalcb by \tfactor
	\qbezier(\grcalcc,\grcalcd)(\grcalcc,\grcalce)(\grcalcb,\grcalcf)
	\advance \grcolumn by 2}

\newcommand{\gibr}{
	\grcalca = \grcolumn
	\multiply \grcalca by \factor
	\advance \grcalca by \hfactor
	\grcalcb = \grcalca
	\advance \grcalcb by \hfactor
	\grcalcc = \grcalca
	\advance \grcalcc by \factor
	\grcalcd = \grrow
	\multiply \grcalcd by \factor
	\grcalce = \grcalcd
	\advance \grcalce by -\tfactor
	\grcalcf = \grcalcd
	\advance \grcalcf by -\hfactor
	\grcalcg = \grcalce
	\advance \grcalcg by -\tfactor
	\grcalch = \grcalcd
	\advance \grcalch by -\factor
	\qbezier(\grcalcc,\grcalcd)(\grcalcc,\grcalce)(\grcalcb,\grcalcf)
	\qbezier(\grcalcb,\grcalcf)(\grcalca,\grcalcg)(\grcalca,\grcalch)
	\advance \grcalcf by -\dfactor
	\advance \grcalcb by \sfactor
	\qbezier(\grcalcc,\grcalch)(\grcalcc,\grcalcg)(\grcalcb,\grcalcf)
	\advance \grcalcf by \sfactor
	\advance \grcalcb by -\tfactor
	\qbezier(\grcalca,\grcalcd)(\grcalca,\grcalce)(\grcalcb,\grcalcf)
	\advance \grcolumn by 2}

\newcommand{\gbrc}{
	\grcalca = \grcolumn
	\multiply \grcalca by \factor
	\advance \grcalca by \hfactor
	\grcalcb = \grcalca
	\advance \grcalcb by \hfactor
	\grcalcc = \grcalca
	\advance \grcalcc by \factor
	\grcalcd = \grrow
	\multiply \grcalcd by \factor
	\grcalce = \grcalcd
	\advance \grcalce by -\tfactor
	\grcalcf = \grcalcd
	\advance \grcalcf by -\hfactor
	\grcalcg = \grcalce
	\advance \grcalcg by -\tfactor
	\grcalch = \grcalcd
	\advance \grcalch by -\factor
	\put(\grcalcb,\grcalcf){\circle{\hfactor}}
	\qbezier(\grcalca,\grcalcd)(\grcalca,\grcalce)(\grcalcb,\grcalcf)
	\qbezier(\grcalcb,\grcalcf)(\grcalcc,\grcalcg)(\grcalcc,\grcalch)
	\advance \grcalcf by -\dfactor
	\advance \grcalcb by -\sfactor
	\qbezier(\grcalca,\grcalch)(\grcalca,\grcalcg)(\grcalcb,\grcalcf)
	\advance \grcalcf by \sfactor
	\advance \grcalcb by \tfactor
	\qbezier(\grcalcc,\grcalcd)(\grcalcc,\grcalce)(\grcalcb,\grcalcf)
	\advance \grcolumn by 2}

\newcommand{\gibrc}{
	\grcalca = \grcolumn
	\multiply \grcalca by \factor
	\advance \grcalca by \hfactor
	\grcalcb = \grcalca
	\advance \grcalcb by \hfactor
	\grcalcc = \grcalca
	\advance \grcalcc by \factor
	\grcalcd = \grrow
	\multiply \grcalcd by \factor
	\grcalce = \grcalcd
	\advance \grcalce by -\tfactor
	\grcalcf = \grcalcd
	\advance \grcalcf by -\hfactor
	\grcalcg = \grcalce
	\advance \grcalcg by -\tfactor
	\grcalch = \grcalcd
	\advance \grcalch by -\factor
	\put(\grcalcb,\grcalcf){\circle{\hfactor}}
	\qbezier(\grcalcc,\grcalcd)(\grcalcc,\grcalce)(\grcalcb,\grcalcf)
	\qbezier(\grcalcb,\grcalcf)(\grcalca,\grcalcg)(\grcalca,\grcalch)
	\advance \grcalcf by -\dfactor
	\advance \grcalcb by \sfactor
	\qbezier(\grcalcc,\grcalch)(\grcalcc,\grcalcg)(\grcalcb,\grcalcf)
	\advance \grcalcf by \sfactor
	\advance \grcalcb by -\tfactor
	\qbezier(\grcalca,\grcalcd)(\grcalca,\grcalce)(\grcalcb,\grcalcf)
	\advance \grcolumn by 2}

\newcommand{\gu}[1]{
	\grcalca = \grcolumn
	\multiply \grcalca by \factor
	\grcalcd = \hfactor
	\multiply \grcalcd by #1
	\advance \grcalca by \grcalcd
	\grcalcb = \grrow
	\advance \grcalcb by -1
	\multiply \grcalcb by \factor
	\put(\grcalca,\grcalcb) {\line(0,1){\hfactor}}
	\advance \grcalcb by \hfactor
	\put(\grcalca,\grcalcb) {\circle*{3}}
	\advance \grcolumn by #1}

\newcommand{\gcu}[1]{
	\grcalca = \grcolumn
	\multiply \grcalca by \factor
	\grcalcd = \hfactor
	\multiply \grcalcd by #1
	\advance \grcalca by \grcalcd
	\grcalcb = \grrow
	\multiply \grcalcb by \factor
	\put(\grcalca,\grcalcb) {\line(0,-1){\hfactor}}
	\advance \grcalcb by -\hfactor
	\put(\grcalca,\grcalcb) {\circle*{3}}
	\advance \grcolumn by #1}

\newcommand{\gmp}[1]{
	\grcalca = \grcolumn
	\multiply \grcalca by \factor
	\advance \grcalca by \hfactor
	\grcalcb = \grrow
	\multiply \grcalcb by \factor
	\put(\grcalca,\grcalcb) {\line(0,-1){\dfactor}}
	\advance \grcalcb by -\factor
	\put(\grcalca,\grcalcb) {\line(0,1){\dfactor}}
	\advance \grcalcb by \hfactor
	\grcalcc = \factor
	\advance \grcalcc by -\qfactor
	\put(\grcalca,\grcalcb) {\circle{\grcalcc}}
	\put(\grcalca,\grcalcb) {\makebox(0,0){$\scriptstyle #1$}}
	\advance \grcolumn by 1}

\newcommand{\gbmp}[1]{
	\grcalca = \grcolumn
	\multiply \grcalca by \factor
	\advance \grcalca by \hfactor
	\grcalcb = \grrow
	\multiply \grcalcb by \factor
	\put(\grcalca,\grcalcb) {\line(0,-1){\dfactor}}
	\advance \grcalcb by -\factor
	\put(\grcalca,\grcalcb) {\line(0,1){\dfactor}}
	\advance \grcalca by -\hfactor
	\advance \grcalca by \dfactor
	\advance \grcalcb by \dfactor
	\grcalcc = \factor
	\advance \grcalcc by -\sfactor
	\put(\grcalca,\grcalcb) {\framebox(\grcalcc,\grcalcc){$\scriptstyle #1$}}
	\advance \grcolumn by 1}

\newcommand{\gbmpt}[1]{
	\grcalca = \grcolumn
	\multiply \grcalca by \factor
	\advance \grcalca by \hfactor
	\grcalcb = \grrow
	\multiply \grcalcb by \factor
	\put(\grcalca,\grcalcb) {\line(0,-1){\dfactor}}
	\advance \grcalcb by -\factor
	\advance \grcalca by -\hfactor
	\advance \grcalca by \dfactor
	\advance \grcalcb by \dfactor
	\grcalcc = \factor
	\advance \grcalcc by -\sfactor
	\put(\grcalca,\grcalcb) {\framebox(\grcalcc,\grcalcc){$\scriptstyle #1$}}
	\advance \grcolumn by 1}

\newcommand{\gbmpb}[1]{
	\grcalca = \grcolumn
	\multiply \grcalca by \factor
	\advance \grcalca by \hfactor
	\grcalcb = \grrow
	\multiply \grcalcb by \factor
	\advance \grcalcb by -\factor
	\put(\grcalca,\grcalcb) {\line(0,1){\dfactor}}
	\advance \grcalca by -\hfactor
	\advance \grcalca by \dfactor
	\advance \grcalcb by \dfactor
	\grcalcc = \factor
	\advance \grcalcc by -\sfactor
	\put(\grcalca,\grcalcb) {\framebox(\grcalcc,\grcalcc){$\scriptstyle #1$}}
	\advance \grcolumn by 1}

\newcommand{\gbmpn}[1]{
	\grcalca = \grcolumn
	\multiply \grcalca by \factor
	\advance \grcalca by \hfactor
	\grcalcb = \grrow
	\multiply \grcalcb by \factor
	\advance \grcalcb by -\factor
	\advance \grcalca by -\hfactor
	\advance \grcalca by \dfactor
	\advance \grcalcb by \dfactor
	\grcalcc = \factor
	\advance \grcalcc by -\sfactor
	\put(\grcalca,\grcalcb) {\framebox(\grcalcc,\grcalcc){$\scriptstyle #1$}}
	\advance \grcolumn by 1}

\newcommand{\glmptb}{
	\grcalca = \grcolumn
	\multiply \grcalca by \factor
	\advance \grcalca by \hfactor
	\grcalcb = \grrow
	\multiply \grcalcb by \factor
	\put(\grcalca,\grcalcb) {\line(0,-1){\dfactor}}
	\advance \grcalcb by -\factor
	\put(\grcalca,\grcalcb) {\line(0,1){\dfactor}}
	\advance \grcalca by -\hfactor
	\advance \grcalca by \dfactor
	\advance \grcalcb by \dfactor
	\put(\grcalca,\grcalcb) {\line(1,0){\factor}}
	\advance \grcalcb by \factor
	\advance \grcalcb by -\sfactor
	\put(\grcalca,\grcalcb) {\line(1,0){\factor}}
	\grcalcc = \factor
	\advance \grcalcc by -\sfactor
	\put(\grcalca,\grcalcb) {\line(0,-1){\grcalcc}}
	\advance \grcolumn by 1}

\newcommand{\glmpt}{
	\grcalca = \grcolumn
	\multiply \grcalca by \factor
	\advance \grcalca by \hfactor
	\grcalcb = \grrow
	\multiply \grcalcb by \factor
	\put(\grcalca,\grcalcb) {\line(0,-1){\dfactor}}
	\advance \grcalca by -\hfactor
	\advance \grcalca by \dfactor
	\advance \grcalcb by -\dfactor
	\put(\grcalca,\grcalcb) {\line(1,0){\factor}}
	\advance \grcalcb by -\factor
	\advance \grcalcb by \sfactor
	\put(\grcalca,\grcalcb) {\line(1,0){\factor}}
	\grcalcc = \factor
	\advance \grcalcc by -\sfactor
	\put(\grcalca,\grcalcb) {\line(0,1){\grcalcc}}
	\advance \grcolumn by 1}

\newcommand{\glmpb}{
	\grcalca = \grcolumn
	\multiply \grcalca by \factor
	\advance \grcalca by \hfactor
	\grcalcb = \grrow
	\multiply \grcalcb by \factor
	\advance \grcalcb by -\factor
	\put(\grcalca,\grcalcb) {\line(0,1){\dfactor}}
	\advance \grcalca by -\hfactor
	\advance \grcalca by \dfactor
	\advance \grcalcb by \dfactor
	\put(\grcalca,\grcalcb) {\line(1,0){\factor}}
	\advance \grcalcb by \factor
	\advance \grcalcb by -\sfactor
	\put(\grcalca,\grcalcb) {\line(1,0){\factor}}
	\grcalcc = \factor
	\advance \grcalcc by -\sfactor
	\put(\grcalca,\grcalcb) {\line(0,-1){\grcalcc}}
	\advance \grcolumn by 1}

\newcommand{\glmp}{
	\grcalca = \grcolumn
	\multiply \grcalca by \factor
	\advance \grcalca by \dfactor
	\grcalcb = \grrow
	\multiply \grcalcb by \factor
	\advance \grcalcb by -\dfactor
	\put(\grcalca,\grcalcb) {\line(1,0){\factor}}
	\advance \grcalcb by -\factor
	\advance \grcalcb by \sfactor
	\put(\grcalca,\grcalcb) {\line(1,0){\factor}}
	\grcalcc = \factor
	\advance \grcalcc by -\sfactor
	\put(\grcalca,\grcalcb) {\line(0,1){\grcalcc}}
	\advance \grcolumn by 1}

\newcommand{\gcmptb}{
	\grcalca = \grcolumn
	\multiply \grcalca by \factor
	\advance \grcalca by \hfactor
	\grcalcb = \grrow
	\multiply \grcalcb by \factor
	\put(\grcalca,\grcalcb) {\line(0,-1){\dfactor}}
	\advance \grcalcb by -\factor
	\put(\grcalca,\grcalcb) {\line(0,1){\dfactor}}
	\advance \grcalca by -\hfactor
	\advance \grcalcb by \dfactor
	\put(\grcalca,\grcalcb) {\line(1,0){\factor}}
	\advance \grcalcb by \factor
	\advance \grcalcb by -\sfactor
	\put(\grcalca,\grcalcb) {\line(1,0){\factor}}
	\advance \grcolumn by 1}

\newcommand{\gcmpt}{
	\grcalca = \grcolumn
	\multiply \grcalca by \factor
	\advance \grcalca by \hfactor
	\grcalcb = \grrow
	\multiply \grcalcb by \factor
	\put(\grcalca,\grcalcb) {\line(0,-1){\dfactor}}
	\advance \grcalcb by -\factor
	\advance \grcalca by -\hfactor
	\advance \grcalcb by \dfactor
	\put(\grcalca,\grcalcb) {\line(1,0){\factor}}
	\advance \grcalcb by \factor
	\advance \grcalcb by -\sfactor
	\put(\grcalca,\grcalcb) {\line(1,0){\factor}}
	\advance \grcolumn by 1}

\newcommand{\gcmpb}{
	\grcalca = \grcolumn
	\multiply \grcalca by \factor
	\advance \grcalca by \hfactor
	\grcalcb = \grrow
	\multiply \grcalcb by \factor
	\advance \grcalcb by -\factor
	\put(\grcalca,\grcalcb) {\line(0,1){\dfactor}}
	\advance \grcalca by -\hfactor
	\advance \grcalcb by \dfactor
	\put(\grcalca,\grcalcb) {\line(1,0){\factor}}
	\advance \grcalcb by \factor
	\advance \grcalcb by -\sfactor
	\put(\grcalca,\grcalcb) {\line(1,0){\factor}}
	\advance \grcolumn by 1}

\newcommand{\gcmp}{
	\grcalca = \grcolumn
	\multiply \grcalca by \factor
	\grcalcb = \grrow
	\multiply \grcalcb by \factor
	\advance \grcalcb by -\factor
	\advance \grcalcb by \dfactor
	\put(\grcalca,\grcalcb) {\line(1,0){\factor}}
	\advance \grcalcb by \factor
	\advance \grcalcb by -\sfactor
	\put(\grcalca,\grcalcb) {\line(1,0){\factor}}
	\advance \grcolumn by 1}

\newcommand{\grmptb}{
	\grcalca = \grcolumn
	\multiply \grcalca by \factor
	\advance \grcalca by \hfactor
	\grcalcb = \grrow
	\multiply \grcalcb by \factor
	\put(\grcalca,\grcalcb) {\line(0,-1){\dfactor}}
	\advance \grcalcb by -\factor
	\put(\grcalca,\grcalcb) {\line(0,1){\dfactor}}
	\advance \grcalca by \hfactor
	\advance \grcalca by -\dfactor
	\advance \grcalcb by \dfactor
	\put(\grcalca,\grcalcb) {\line(-1,0){\factor}}
	\advance \grcalcb by \factor
	\advance \grcalcb by -\sfactor
	\put(\grcalca,\grcalcb) {\line(-1,0){\factor}}
	\grcalcc = \factor
	\advance \grcalcc by -\sfactor
	\put(\grcalca,\grcalcb) {\line(0,-1){\grcalcc}}
	\advance \grcolumn by 1}

\newcommand{\grmpt}{
	\grcalca = \grcolumn
	\multiply \grcalca by \factor
	\advance \grcalca by \hfactor
	\grcalcb = \grrow
	\multiply \grcalcb by \factor
	\put(\grcalca,\grcalcb) {\line(0,-1){\dfactor}}
	\advance \grcalca by \hfactor
	\advance \grcalca by -\dfactor
	\advance \grcalcb by -\dfactor
	\put(\grcalca,\grcalcb) {\line(-1,0){\factor}}
	\advance \grcalcb by -\factor
	\advance \grcalcb by \sfactor
	\put(\grcalca,\grcalcb) {\line(-1,0){\factor}}
	\grcalcc = \factor
	\advance \grcalcc by -\sfactor
	\put(\grcalca,\grcalcb) {\line(0,1){\grcalcc}}
	\advance \grcolumn by 1}

\newcommand{\grmpb}{
	\grcalca = \grcolumn
	\multiply \grcalca by \factor
	\advance \grcalca by \hfactor
	\grcalcb = \grrow
	\multiply \grcalcb by \factor
	\advance \grcalcb by -\factor
	\put(\grcalca,\grcalcb) {\line(0,1){\dfactor}}
	\advance \grcalca by \hfactor
	\advance \grcalca by -\dfactor
	\advance \grcalcb by \dfactor
	\put(\grcalca,\grcalcb) {\line(-1,0){\factor}}
	\advance \grcalcb by \factor
	\advance \grcalcb by -\sfactor
	\put(\grcalca,\grcalcb) {\line(-1,0){\factor}}
	\grcalcc = \factor
	\advance \grcalcc by -\sfactor
	\put(\grcalca,\grcalcb) {\line(0,-1){\grcalcc}}
	\advance \grcolumn by 1}

\newcommand{\grmp}{
	\grcalca = \grcolumn
	\multiply \grcalca by \factor
	\advance \grcalca by \factor
	\advance \grcalca by -\dfactor
	\grcalcb = \grrow
	\multiply \grcalcb by \factor
	\advance \grcalcb by -\dfactor
	\put(\grcalca,\grcalcb) {\line(-1,0){\factor}}
	\advance \grcalcb by -\factor
	\advance \grcalcb by \sfactor
	\put(\grcalca,\grcalcb) {\line(-1,0){\factor}}
	\grcalcc = \factor
	\advance \grcalcc by -\sfactor
	\put(\grcalca,\grcalcb) {\line(0,1){\grcalcc}}
	\advance \grcolumn by 1}
\newcommand{\gsy}{
	\grcalca = \grcolumn
	\multiply \grcalca by \factor
	\advance \grcalca by \hfactor
	\grcalcb = \grcalca
	\advance \grcalcb by \hfactor
	\grcalcc = \grcalca
	\advance \grcalcc by \factor
	\grcalcd = \grrow
	\multiply \grcalcd by \factor
	\grcalce = \grcalcd
	\advance \grcalce by -\tfactor
	\grcalcf = \grcalcd
	\advance \grcalcf by -\hfactor
	\grcalcg = \grcalce
	\advance \grcalcg by -\tfactor
	\grcalch = \grcalcd
	\advance \grcalch by -\factor
	\qbezier(\grcalcc,\grcalcd)(\grcalcc,\grcalce)(\grcalcb,\grcalcf)
	\qbezier(\grcalcb,\grcalcf)(\grcalca,\grcalcg)(\grcalca,\grcalch)
	\advance \grcalcf by -\dfactor
	\advance \grcalcb by \sfactor
	\qbezier(\grcalcc,\grcalch)(\grcalcc,\grcalcg)(\grcalcb,\grcalcf)
	\qbezier(\grcalca,\grcalcd)(\grcalca,\grcalce)(\grcalcb,\grcalcf)
	\advance \grcolumn by 2}

\newcommand{\gwmuh}[3]{
	\grcalca = \grcolumn
	\multiply \grcalca by \factor
	\grcalcb = #2
	\advance \grcalcb by #3
	\multiply \grcalcb by \qfactor
	\advance \grcalca by \grcalcb
	\grcalcb = \grrow
	\multiply \grcalcb by \factor
	\grcalcc = #3
	\advance \grcalcc by -#2
	\multiply \grcalcc by \hfactor
	\grcalcd = \factor
	\advance \grcalcd by \hfactor
	\put(\grcalca,\grcalcb){\oval(\grcalcc,\grcalcd)[b]}
	\grcalca = \grcolumn
	\multiply \grcalca by \factor
	\grcalcc = #1
	\multiply \grcalcc by \hfactor
	\advance \grcalca by \grcalcc
	\advance \grcalcb by -\hfactor
	\advance \grcalcb by -\qfactor
	\put(\grcalca,\grcalcb) {\line(0,-1){\qfactor}}
	\advance \grcolumn by #1}

\newcommand{\gwcmh}[3]{
	\grcalca = \grcolumn
	\multiply \grcalca by \factor
	\grcalcb = #2
	\advance \grcalcb by #3
	\multiply \grcalcb by \qfactor
	\advance \grcalca by \grcalcb
	\grcalcb = \grrow
	\advance \grcalcb by -1
	\multiply \grcalcb by \factor
	\grcalcc = #3
	\advance \grcalcc by -#2
	\multiply \grcalcc by \hfactor
	\grcalcd = \factor
	\advance \grcalcd by \hfactor
	\put(\grcalca,\grcalcb){\oval(\grcalcc,\grcalcd)[t]}
	\grcalca = \grcolumn
	\multiply \grcalca by \factor
	\grcalcc = #1
	\multiply \grcalcc by \hfactor
	\advance \grcalca by \grcalcc
	\advance \grcalcb by \factor
	\put(\grcalca,\grcalcb) {\line(0,-1){\qfactor}}
	\advance \grcolumn by #1}

\newcommand{\gsbox}[1]{
	\grcalca = \grcolumn
	\multiply \grcalca by \factor
	\grcalcb = \grrow
	\multiply \grcalcb by \factor
	\advance \grcalcb by -\factor
	\grcalcc = #1
	\multiply \grcalcc by \factor
	\grcalcd = \factor
	\put(\grcalca,\grcalcb){\framebox(\grcalcc,\grcalcd){}}}

\newcommand{\gbox}[2]{
	\grcalca = \grcolumn
	\multiply \grcalca by \factor
	\grcalcb = \grrow
	\multiply \grcalcb by \factor
	\advance \grcalcb by -\factor
	\grcalcc = #1
	\multiply \grcalcc by \factor
	\grcalcd = #2
	\multiply \grcalcd by \factor
	\put(\grcalca,\grcalcb){\framebox(\grcalcc,\grcalcd){}}}

\newcommand{\linea}{\gcl{1}}
\newcommand{\cruce}{\gbr}

\newcommand{\gmuc}{
	\grcalca = \grcolumn
	\advance \grcalca by 1
	\multiply \grcalca by \factor
	\grcalcb = \grrow
	\multiply \grcalcb by \factor
	\grcalcc = \factor
	\advance \grcalcc by \hfactor
	\put(\grcalca,\grcalcb){\oval(\factor,\grcalcc)[b]}
	\advance \grcalcb by -\hfactor
	\advance \grcalcb by -\qfactor
	\put(\grcalca,\grcalcb){\circle{\hfactor}}
	\put(\grcalca,\grcalcb) {\line(0,-1){\qfactor}}
	\advance \grcolumn by 2}

\begin{document}
	%
	%
	%
	\title{Non-associative Solomon's descent algebras}
	\author{J. M. P\'erez-Izquierdo}%
	\address{Dpto. Matem\'aticas y Computaci\'on \\ Edificio CCT - C/ Madre de Dios, 53 \\ Universidad de La Rioja \\ 26006, Logro\~no, Spain}
	\email{jm.perez@unirioja.es}
	\thanks{}
	\dedicatory{In memory of Vicente R. Varea, a teacher who touched our hearts and inspired our minds.
	}%
	%
	\begin{abstract}
		Motivated by the observation that universal enveloping algebras of relatively free Sabinin algebras admit natural associative internal products, we revisit the classical relationship between Solomon's descent algebra, noncommutative symmetric functions, and Patras' descent algebras of graded bialgebras from a non-associative perspective. Using the language of algebras over cocommutative connected Hopf operads as a unifying framework, we describe a setting in which these connections extend beyond the associative case. This viewpoint allows us to interpret the free non-associative algebra as a non-associative analogue of Solomon’s descent algebra inside a suitably defined non-associative Malvenuto--Reutenauer-type Hopf algebra. We further provide a detailed classification of this object as an algebra, as a coalgebra, and as an associative algebra with respect to its internal product.
	\end{abstract}
	\maketitle

%
\section{Introduction}

Solomon’s descent algebra was originally introduced in the context of finite Coxeter groups as a subalgebra encoding information on descent sets with respect to a fixed system of simple reflections  \cite{MR0444756}. In that general setting, Solomon proved that the linear span of elements with prescribed descent data is closed under multiplication, a result that may be viewed as a Mackey-type formula for Coxeter groups. The case of the symmetric group $\bS_n$, corresponding to the Coxeter system of type~$A$, subsequently became prominent due to its particularly rich combinatorial structure and its close connections with compositions, permutation statistics, and representation theory.

In type~$A$, Solomon’s descent algebra admits several reformulations that have played a central role in algebraic combinatorics. Notably, it appears naturally as a Hopf subalgebra of the Malvenuto--Reutenauer Hopf algebra of permutations \cite{MR1358493}, and is closely related, via graded duality, to the Hopf algebras of noncommutative and quasisymmetric functions \cite{MR1327096}. These developments led to further generalizations, including the descent algebras of graded bialgebras introduced by Patras \cite{MR1302855}, and their reinterpretation within the broader theory of combinatorial Hopf algebras \cite{MR2196760}. A common feature of these constructions is the presence of an additional associative operation, often referred to as an internal product, which interacts in a nontrivial way with the bialgebra structure.

The aim of this work is to examine how aspects of the descent algebra formalism extend beyond the associative case. Our motivation is the observation that universal enveloping algebras of Sabinin algebras, although generally non-associative as algebras, admit natural associative internal products. This phenomenon, which reflects the interaction between algebraic operations and the coalgebra structure, suggests that part of the role played by descent algebras and the combinatorics of symmetric functions might persist beyond the associative and commutative framework, when the usual symmetries are replaced by more general ones, such as those associated with Moufang or Bol-type structures. From this point of view, classical techniques based on convolution algebras and symmetric functions, such as the construction of Lie idempotents, might also have analogues in non-associative situations, including those arising from Sabinin-type structures.

More concretely, we show that the descent algebra formalism and the associated convolution techniques can be realized in certain non-associative contexts, and that in the free case this realization admits an explicit combinatorial model.

We place our discussion within the existing framework of algebras over cocommutative connected Hopf operads. This language has proved effective in organizing a wide range of constructions in combinatorial Hopf algebra theory, and here it serves mainly as a convenient formal setting. 

The paper is organized as follows. Section~2 introduces notation and basic combinatorial objects. In Section~3, we construct the relatively free $\propP$-algebra associated with a coalgebra and establish the Hopf-algebraic framework. Section~4 explains the relations with Hopf algebras of symmetric functions and descent algebras, leading to non-associative analogues of Solomon-type constructions (Corollaries~\ref{cor:main} and~\ref{cor:internal_product_on_P-Sol}). Section~5 specializes the general theory to universal enveloping algebras of Sabinin algebras (Corollary~\ref{cor:main_Sabinin}), which originally motivated this work. The abstract framework developed above admits a particularly concrete realization in the case of the free non-associative algebra, which we analyze in detail in Section~6. We construct a non-associative Malvenuto--Reutenauer-type Hopf algebra and describe its algebraic (Theorem~\ref{thm:classification_*}), coalgebraic (Corollary~\ref{cor:classification_Delta}), and internal product structures (Proposition~\ref{prop:quotient_compermutations_to_compositions} and Theorem~\ref{thm:structure_inner}). An appendix collects explicit low-degree examples of primitive elements.

\section{Notation relative to operads and pseudo-compositions}

This section establishes the notational conventions used throughout the paper.  The underlying categories required in this work are $\Vect$ (the category of $\field$-vector spaces and linear maps over a base field $\field$ of characteristic zero), $\grVect$ (the category of $\mathbb{N}$-graded $\field$-vector spaces and $\field$-linear maps that preserve the grading), and $\Smod$. Objects in $\Smod$, or $\bS$-modules, are $\field$-vector spaces $M = \bigoplus_{n \geq 0} M(n)$, where each $M(n)$ is a right $\bS_n$-module for the symmetric group $\bS_n$ of degree $n$. Morphisms in $\Smod$ are $\bS_n$-equivariant linear maps that preserve the grading. In what follows, $\Und$ denotes either $\grVect$ or $\Smod$ interchangeably. 

We assume the reader is familiar with the following monoidal structures. For objects $M, N$ in $\grVect$, the tensor product is given by:
\begin{equation*}
	(M \otimes N)(n) := \bigoplus_{i=0}^n M(i) \otimes N(n-i),
\end{equation*}
while for objects $P, Q$ in $\Smod$, it is defined as:
\begin{equation*}
	(P \otimes Q)(n) := \bigoplus_{i=0}^n P(i) \otimes Q(n-i) \otimes_{\bS_i \times \bS_{n-i}} \field \bS_n.
\end{equation*}
We also assume familiarity with algebras over symmetric operads in these categories, including Hopf operads, in particular as developed in \cites{MR2423811, MR2674654}. Throughout this paper, we adopt the following conventions:
\begin{itemize}
	\item $\propP$ denotes a cocommutative connected Hopf symmetric operad;
	\item coalgebras are always assumed to be coassociative and cocommutative;
	\item the base field $\field$ has characteristic zero.
\end{itemize}
In particular, we will always assume $\propP(0) = \field 1_0$. Furthermore, we shall freely write either $p(v_1, \dots, v_m)$ or $p(v_1 \otimes \dots \otimes v_m)$ for the action of $p \in \propP(m)$ on elements $v_1, \dots, v_m$ of a $\propP$-algebra in $\grVect$. Occasionally, we use the notation $\underline{\otimes}$ to emphasize that a tensor product is balanced with respect to a group action.

A pseudo-composition $\pi$ of a set $A$, denoted by $\pi \models A$, is a tuple $\pi = (\pi_1, \dots, \pi_r)$ of sets (the parts of $\pi$) such that $A$ is the disjoint union of $\pi_1, \dots, \pi_r$. Empty parts are allowed. The type of $\pi$ is $\typ(\pi) = (|\pi_1|, \dots, | \pi_r |)$, where $|\cdot|$ denotes cardinality. The length of $\pi$ is $\longi(\pi) := r$, and its weight is $|\pi| = |\pi_1| + \dots + |\pi_r|$. 

Let $\comp^\#(A)$ denote the set of all pseudo-compositions of $A$, and let $\comp(A)$ be the set of all compositions (i.e., pseudo-compositions with no empty parts). For $n \geq 1$, we let $[n]$ denote the set $\{1, \dots, n\}$, and define $\comp^\#_n = \comp^\#([n])$ and $\comp_n = \comp([n])$. 

For integers, a pseudo-composition of $n$, denoted by $I \models n$, is a tuple $I = (i_1, \dots, i_r)$ of non-negative integers such that $n = i_1 + \dots + i_r$; its length is likewise $\longi(I) := r$. If all $i_j$ are positive, $I$ is a composition of $n$. A partition $P = \{ i_1, \dots, i_r \}$ of $n$, denoted by $P \vdash n$, is the unordered analogue of a composition. As usual, we can represent a partition by a tuple $(i_1, \dots, i_r)$ such that $i_1 \geq \dots \geq i_r$. The symmetric group $\bS_r$ acts on pseudo-compositions of length $r$ via:
\[
\tau \pi := ( \pi_{\tau^{-1}(1)}, \dots, \pi_{\tau^{-1}(r)})=: \pi \tau^{-1} \quad \text{and} \quad \tau I := ( i_{\tau^{-1}(1)}, \dots, i_{\tau^{-1}(r)}):= I \tau^{-1}.
\]

Standardization maps $\pi \in \comp^\#(A)$ (for $A \subseteq [n]$) to $\st(\pi) \in \comp^\#_{|A|}$ by re-indexing the elements of $\pi$ to preserve their relative order. Conversely, the operator $\utimes$ combines $\pi \in \comp^\#_n$ and $\pi' \in \comp^\#_m$ to produce $\pi \utimes \pi' \in \comp^\#_{n+m}$ by shifting the elements of $\pi'$ by $n$ and concatenating the result to $\pi$. For example:
\begin{align*}
	\st((\{3,8\}, \{2,4\},\{ 6\})) &= (\{2,5\},\{1,3\},\{4\}), \\
	(\{2,3\},\{1,4\}) \utimes (\{3\},\{1,2\}) &= (\{2,3\},\{1,4\}, \{7\}, \{5,6\}).
\end{align*}

The restriction of $\pi = (\pi_1, \dots, \pi_r) \in \comp^\#(A)$ to a subset $S \subseteq A$ is defined as:
\begin{equation*}
	\pi|_S := (\pi_1 \cap S, \dots, \pi_r \cap S).
\end{equation*}
For instance, for $p = 2$, $m = 6$, and $\pi = (\{2,3\}, \{4,5\}, \{1,6\})$, we have:
\begin{align*}
	\pi|_{[p]} \otimes \st(\pi_{[m] \setminus [p]}) &= (\{2\}, \emptyset, \{1\}) \otimes \st(\{3\}, \{4,5\}, \{6\}) \\
	&= (\{ 2\}, \emptyset, \{1\}) \otimes (\{1\}, \{2,3\}, \{4\}).
\end{align*}

Permutations $\sigma \in \bS_n$ are written as tuples $(\sigma(1), \dots, \sigma(n))$. A shuffle $\Sh(\pi) = \Sh(\pi_1, \dots, \pi_r)$ is associated with any pseudo-composition $\pi$ of $[n]$. For example:
\begin{equation*}
	\Sh(\{2,3\}, \{1,4\}, \{7\}, \{5,6\}) = (2,3,1,4,7,5,6)^{-1}.
\end{equation*}

Given $\typ(\pi)$, the composition $\pi$ is uniquely determined by $\Sh(\pi)$. A shuffle is said to have type $I$ if it can be written as $\Sh(\pi)$ for some $\pi$ of type $I$. We denote by $\Sh_{i_1,\dots,i_r}$ the set of all shuffles in $\bS_n$ of type $I=(i_1,\dots,i_r)$, and set
\begin{displaymath}
q_{i_1,\dots,i_r} := \sum_{\sigma \in \Sh_{i_1,\dots,i_r}} \sigma.
\end{displaymath}

\section{A relatively free $\propP$-algebra $\propP_\Omega(C)$ associated to a coalgebra} 
\label{sec:relatively free}

The connection of this paper with symmetric functions stems from the fact that these form a bialgebra freely generated, as an algebra, by a coalgebra encoding much of their combinatorics. Working with such structures naturally leads beyond the operadic framework. While operads suffice to describe purely algebraic operations, the identities of interest in the theory of loops involve a nontrivial interaction between multiplication and comultiplication, such as the Moufang--Hopf identity:
\begin{equation}\label{eq:Moufang-Hopf}
	\sum x_{(1)}(y(x_{(2)} z)) = \sum ((x_{(1)} y) x_{(2)}) z,
\end{equation}
represented by the diagram:
\begin{displaymath}
	\gbeg{4}{9}
	\gcmu \gcl{1} \gcl{1} \gnl
	\gcl{1} \gbr \gcl{1} \gnl
	\gcl{1} \gcl{1} \gmu \gnl
	\gcl{1} \gcl{1} \gcn{1}{1}{2}{1} \gnl
	\gcl{1} \gmu \gnl
	\gcl{1} \gcn{1}{1}{2}{1} \gnl
	\gmu \gnl
	\gend
	=
	\gbeg{4}{9}
	\gcmu \linea \linea \gnl
	\linea \cruce \linea \gnl
	\gmu \linea \linea \gnl
	\gcn{2}{1}{2}{3} \linea \linea \gnl
	\gvac{1} \gmu \linea \gnl
	\gvac{1} \gcn{2}{1}{2}{3} \linea \gnl
	\gvac{2} \gmu \gnl
	\gend
\end{displaymath}
As a consequence, although an operadic formulation is possible, these identities inherently involve both operations and cooperations, and are therefore more naturally expressed in the language of props. Fortunately, bialgebras generated as algebras by coalgebras can be viewed as algebras (in the category of coalgebras) over suitable operads.

\subsection{A controlling prop $\propH$ and a controlling operad $\propP_\Omega$} 
We adopt the convention that for elements $h, h'$ in a prop, if their composition is defined, $hh'$ denotes vertical composition and $h \otimes h'$ denotes horizontal composition \cite{MR171826}.

Let $\propP$ be a cocommutative connected Hopf symmetric operad and let $\propH$ be the prop generated by $\propP$ and the cooperations $\delta_2 \in \propH(2,1)$, $\delta_1 \in \propH(1,1)$, and $\delta_0 \in \propH(0,1)$ (i.e., one input and two outputs, etc.), subject to the following relations:
\begin{itemize}
	\item[(R1)] \textit{Coassociativity}: $(\delta_2 \otimes \delta_1) \delta_2 = (\delta_1 \otimes \delta_2) \delta_2$.
	\item[(R2)] \textit{Cocommutativity}: $(2,1) \delta_2 = \delta_2$.
	\item[(R3)] \textit{Counitality}: $(\delta_0 \otimes \delta_1) \delta_2 = \delta_1 = (\delta_1 \otimes \delta_0) \delta_2$.
	\item[(R4)] \textit{Distributive laws}: for any $p \in \propP(n)$ ($n \geq 0$),
	\begin{displaymath}
		\delta_2 p = \sum (p_{(1)} \otimes p_{(2)}) \sigma_{n,n} \delta^{\otimes n}_2 \quad \text{and} \quad \delta_0 p = \epsilon(p) \delta^{\otimes n}_0.
	\end{displaymath}
	\item[(R5)] $\delta_1 = 1_1 \in \propP(1)$.
\end{itemize}
Here, $\sigma_{n,n} = (1, n+1, 2, n+2, \dots, n, 2n)$, while $p \mapsto \sum p_{(1)} \otimes p_{(2)}$ and $p \mapsto \epsilon(p)$ denote the comultiplication and counit on $\propP(n)$, respectively. Consequently, $\propH$-algebras are equivalent to cocommutative Hopf $\propP$-algebras.

We extend the Hopf structure of $\propP$ to $\propH$ by defining, on generators:
\begin{align*}
	\delta(p) &= \sum p_{(1)} \otimes p_{(2)}, & \delta(\delta_2) &= \delta_2 \otimes \delta_2, & \delta(\delta_0) &= \delta_0 \otimes \delta_0, \\
	\epsilon(p) &= \epsilon(p), & \epsilon(\delta_2) &= 1, & \epsilon(\delta_0) &= 1,
\end{align*}
and on permutations, horizontal, and vertical compositions:
\begin{align*}
	\delta(\sigma) &= \sigma \otimes \sigma, & \delta(hh') &= \delta(h)\delta(h'), & \delta(h \otimes h') &= \sum (h_{(1)} \otimes h'_{(1)}) \otimes (h_{(2)} \otimes h'_{(2)}), \\
	\epsilon(\sigma) &= 1, & \epsilon(hh') &= \epsilon(h) \epsilon(h'), & \epsilon(h \otimes h') &= \epsilon(h) \epsilon(h').
\end{align*}
The compatibility of $\delta$ with the distributive laws follows from the cocommutativity of $\propP$. For example, $\delta(\delta_2 p) = \delta(\delta_2)\delta(p)$, and expanding both sides shows they agree. The following proposition summarizes the properties of $\propH$, where $\propCom^c$ denotes the commutative cooperad.

\begin{proposition} \label{prop:Omega}
	\begin{enumerate}
		\item For any $m, n \geq 0$, we have the decomposition:
		\begin{displaymath}
			\propH(n,m) = \sum_{r \geq 0} \propP^{\otimes n}(r) \bS_r \propCom^{c\otimes m}(r).
		\end{displaymath}
		\item For any $m, n \geq 0$, the subspace $(\propH(n,m), \delta, \epsilon)$ is a coalgebra in $\Vect_{\field}$.
		\item $\propH(1,*) = \bigoplus_{m \geq 0} \propH(1,m)$ is a cocommutative Hopf operad.
		\item For any $n \geq 1$, the elements $\Delta_n = \sigma_{n,n}(\delta_2 \otimes \dots \otimes \delta_2) \in \propH(2n,n)$ satisfy $\Delta_n h = \delta(h) \Delta_m$ for all $h \in \propH(n,m)$.
		\item Let $\Omega = \bigcup_{k \geq 0} \Omega_k$, where each $\Omega_k$ spans a coideal of $\propH(1,k)$. Then the subspace
		\begin{equation}
			\tilde{\Omega}_m := \sum_{k \geq 0, j \geq 1} \propH(1,j) (1_1 \otimes \dots \otimes \Omega_k \otimes \dots \otimes 1_1) \propH(j-1 + k, m)
		\end{equation}
		is a coideal of $\propH(1,m)$. The sum $\tilde{\Omega} = \bigoplus_{m \geq 0} \tilde{\Omega}_m$ is an ideal of the operad $\propH(1,*)$, and the quotient $\propP_\Omega = \propH(1,*)/\tilde{\Omega}$ is a cocommutative Hopf operad.
	\end{enumerate}
\end{proposition}

The set $\Omega$, and its extension to the ideal $\tilde{\Omega}$, collects identities involving the operadic operations in $\propP$ as well as the comultiplication $\delta_2$ and counit $\delta_0$. However, $\propP_\Omega$ remains an operadic structure. Note that since $h(1_1, \dots, 1_0, \dots, 1_1) \in \tilde{\Omega}$ whenever $h \in \tilde{\Omega}$, the ideal $\tilde{\Omega}$ accounts for the existence of unit elements in the algebras.

\subsection{The relatively free $\propP$-algebra $\propP_\Omega(C)$} 
Let $C' = (\bigoplus_{n \geq 1} C_n, \Delta')$ be a non-counital cocommutative coalgebra in $\Und$. Since $\field 1_0 = \propP(0) = (\propP \circ C')(0)$, we identify $\field 1_1 \otimes C'(n)$ with $C'(n)$ and define:
\begin{displaymath}
	C := \bigoplus_{n \geq 0} C_n := \field 1_0 \oplus \bigoplus_{n \geq 1} C'(n) \subseteq \propP \circ C'.
\end{displaymath}
We endow $C$ with the operations:
\begin{align*}
	\Delta(c) &:= 1_0 \otimes c + c \otimes 1_0 + \Delta'(c), & \Delta(1_0) &:= 1_0 \otimes 1_0, \\
	\epsilon(1_0) &:= 1, & \epsilon(C') &:= 0,
\end{align*}
for $c \in C'$, resulting in a cocommutative counital coalgebra $(C, \Delta, \epsilon)$ that generates the free $\propP$-algebra $\propP \circ C'$. These maps extend to $\propP$-algebra homomorphisms $\Delta \colon \propP \circ C' \longrightarrow (\propP \circ C') \otimes (\propP \circ C')$ and $\epsilon \colon \propP \circ C' \longrightarrow \field$, making $\propP \circ C'$ an $\propH$-algebra.

\begin{proposition}
	Let $C'$ be a non-counital cocommutative coalgebra in $\Und$. Then $\propP \circ C'$ is a cocommutative connected Hopf $\propP$-algebra. Moreover, for any $\tilde{\Omega}$ as defined above, the space
	\begin{displaymath}
		\tilde{\Omega}(C) := \spann_{\field}\{ h(c_1, \dots, c_m) \mid h \in \tilde{\Omega}_m, c_i \in C \}
	\end{displaymath}
	is an ideal and a coideal of $\propP \circ C'$.
\end{proposition}

\begin{proof}
	To see that $\propP \circ C'$ is a cocommutative Hopf $\propP$-algebra, it suffices to verify that the required $\propP$-algebra homomorphisms agree on the generators $C'$. Connectivity follows from $(\propP \circ C')(0) = \field 1_0$. 
	By definition, $\tilde{\Omega}(C)$ is an ideal. For $h \in \tilde{\Omega}_m$, we have $\Delta(h(c)) = (\delta_2 h)(c) = \delta(h) \Delta(c)$. Since $\tilde{\Omega}$ is a coideal and the comultiplication on $C^{\otimes m}$ is compatible, the result follows.
\end{proof}

Later, in Section~\ref{sec:example}, we will freely reuse the notation
\[
\Omega(\propP \circ C') := \{\, h(u_1,\dots, u_m) \mid h \in \Omega,\ u_i \in \propP \circ C' \,\}.
\]

We define the cocommutative connected Hopf $\propP$-algebra associated with $C'$ as:
\begin{equation}\label{eq:C_relatively_free}
	\propP_{\Omega}(C) := (\propP \circ C') / \tilde{\Omega}(C).
\end{equation}
Note that while $\propP_\Omega(C)$ is a Hopf $\propP_\Omega$-algebra, it may not coincide with $\propP_\Omega \circ C'$, as $\propP_\Omega$ can be significantly larger than $\propP$. Nevertheless, $\propP_\Omega(C)$ provides a more manageable framework.

\begin{proposition}
	Let $H$ be a cocommutative connected Hopf $\propP$-algebra in $\Und$, and let $\varphi \colon C \longrightarrow H$ be a homomorphism of cocommutative connected coalgebras. If $\tilde{\Omega}(\varphi(C)) = \{0\}$, then $\varphi$ extends uniquely to a homomorphism $\varphi \colon \propP_{\Omega}(C) \longrightarrow H$ of connected Hopf $\propP$-algebras.
\end{proposition}

\section{Relations with Hopf algebras of symmetric functions and descent algebras}
As previously noted, our goal is to contextualize the associative internal product that naturally appears in certain universal enveloping algebras of Sabinin algebras, which are generally highly non-associative. Since these enveloping algebras are generated by coalgebras—just as Hopf algebras of symmetric functions \cite{MR0506405}—one may expect the latter to play a role in this setting. The purpose of this section is to make this connection precise. While most of the results are non-associative counterparts of known facts, they show that the basic constructions involving symmetric functions and descent algebras remain meaningful.

\subsection{The Hopf $\propP$-algebra of symmetric functions}

We begin by recalling the standard coalgebra underlying symmetric functions. Let $E'$ be the coalgebra in $\grVect$ defined by
\begin{displaymath}
	E' := \bigoplus_{n \geq 1} \field e_n, \qquad \Delta'(e_1) := 0, \qquad \Delta'(e_n) := \sum_{i=1}^{n-1} e_i \otimes e_{n-i}.
\end{displaymath}
Extending by $e_0 := 1_0$, we obtain a counital coalgebra $E := \field 1_0 \oplus E'$ with
\begin{displaymath}
	\Delta(e_n) := \sum_{i=0}^n e_i \otimes e_{n-i}.
\end{displaymath}

\begin{definition}
	The Hopf $\propP$-algebra of symmetric functions satisfying the identities $\Omega$ (or $\tilde{\Omega}$) is defined as
	\[
	\propP_\Omega(E) := \propP \circ E' / \tilde{\Omega}(E).
	\]
\end{definition}

\begin{remark}
	By \cite{MR2674654}, $\propP_\Omega(E)$ is a cocommutative Hopf $\propP$-algebra.
\end{remark}

\begin{example}
When $\tilde{\Omega} = \{0\}$ and $\propP$ is the commutative or associative operad, one recovers the classical Hopf algebras of symmetric and noncommutative symmetric functions. The examples we have in mind, however, lie closer to the non-associative setting corresponding to $\propP = \Mag$, with $\Omega = \Omega_3$ given by the Moufang-Hopf identity~\eqref{eq:Moufang-Hopf}.
\end{example}

We now describe the descent algebra associated with $\propP_\Omega(E)$ and show that it recovers the whole structure. Let $\End_{\field}(\propP_\Omega(E))$ denote the space of linear endomorphisms. The subspace of degree-preserving maps with finite support is a $\propP$-algebra under the convolution product
\begin{displaymath}
	p(f_1,\dots,f_n)(u) := \sum p(f_1(u_{(1)}),\dots,f_n(u_{(n)})).
\end{displaymath}
In particular, the projections $\Id_n$ onto $\propP_\Omega(E)(n)$ define a graded subcoalgebra with
\begin{displaymath}
	\Delta(\Id_n) := \sum_{i=0}^n \Id_i \otimes \Id_{n-i}.
\end{displaymath}

The $\propP$-subalgebra generated by the $\Id_n$ is the descent algebra $\Sigma^{\propP_\Omega(E)}$. The assignment $e_n \mapsto \Id_n$ defines a morphism
\begin{equation} \label{eq:isomorphism_descent}
	\varphi \colon \propP \circ E' \longrightarrow \Sigma^{\propP_\Omega(E)}.
\end{equation}

The following result provides an interpretation of $\propP_\Omega(E)$ in terms of Patras' descent algebras.

\begin{theorem} \label{thm:isomorphism_symmetric_convolution}
	The map $\varphi$ induces an isomorphism of $\propP$-algebras
	\[
	\propP_\Omega(E) \cong \Sigma^{\propP_\Omega(E)},
	\qquad [e_n] \longmapsto \Id_n.
	\]
\end{theorem}

Before proving the theorem, we establish two technical lemmas.

\begin{lemma} \label{lem:evaluation_f}
	For any $h \in \propH(1,m)$, $e_{k_1},\dots,e_{k_m} \in E$, and $v \in \propP_\Omega(E)$,
	\[
	\varphi(h(e_{k_1},\dots,e_{k_m}))(v)
	= \sum h(\Id_{k_1}(v_{(1)}),\dots,\Id_{k_m}(v_{(m)})).
	\]
\end{lemma}

\begin{proof}
	We may assume $h = p d$, with $p \in \propP(r)$ and $d \in \propCom^{c\otimes m}(r)$. Then
	\begin{align*}
		\varphi(h(e_{k_1},\dots,e_{k_m}))(v)
		&= p(\varphi^{\otimes r}(d(e_{k_1}\otimes\cdots\otimes e_{k_m})))(v) \\
		&= \sum p\big(d(\Id_{k_1}(v_{(1)})\otimes\cdots\otimes\Id_{k_m}(v_{(m)}))\big),
	\end{align*}
	which yields the result.
\end{proof}

\begin{lemma} \label{lem:Mackey}
	Let $I=(i_1,\dots,i_r)$ and $K=(k_1,\dots,k_m)$. For $h \in \propH(1,m)$ and $p \in \propP(r)$,

\begin{align*}
	\varphi(h(e_{k_1}, \dots, &e_{k_m})) \varphi(p(e_{i_1}, \dots, e_{i_r})) \\
	& = \sum_{(j_{l,q})} \varphi(h(p_{(1)}(e_{j_{1,1}}, \dots, e_{j_{r,1}}), \dots, p_{(m)}(e_{j_{1,m}}, \dots, e_{j_{r,m}}))),
\end{align*}
	where $(j_{l,q})$ ranges over matrices with non-negative integer entries such that the $l$-th row sum is $i_l$ and the $q$-th column sum is $k_q$. 
\end{lemma}

\begin{proof}
	Evaluating on $v$ and using Lemma~\ref{lem:evaluation_f}, one obtains the desired expression by cocommutativity.
\end{proof}

\begin{proof}[Proof of Theorem~\ref{thm:isomorphism_symmetric_convolution}]	
	We determine the kernel of $\varphi$. For $u \in \propP \circ E'(n)$, we have
	\[
	\varphi(u)([e_n]) = u + \tilde{\Omega}(E),
	\]
	which shows that $\ker \varphi \subseteq \tilde{\Omega}(E)$.
	
	Conversely, let $h(e_{k_1},\dots,e_{k_m}) \in \tilde{\Omega}(E)$. Then, by Lemma~\ref{lem:Mackey}, for any $u \in \propP \circ E'(n)$ we have
	\[
	\varphi(h(e_{k_1},\dots,e_{k_m}))([u])
	= \varphi(h(e_{k_1},\dots,e_{k_m})) \, \varphi(u)([e_n]) = [0],
	\]
	since elements of $\varphi(\tilde{\Omega}(E))$ annihilate $[e_n]$. This shows that $\tilde{\Omega}(E) \subseteq \ker \varphi$.
\end{proof}

Via this identification, the composition in $\Sigma^{\propP_\Omega(E)}$ induces an internal associative product $\circ$ on $\propP_\Omega(E)$.

\begin{corollary}[Mackey formula]
For $n \ge 0$, the space $\propP_\Omega(E)(n)$ becomes an associative algebra with product
	\begin{align*}
	[q(e_{k_1},\dots,e_{k_m})]\circ[p(e_{i_1},\dots,e_{i_r})] & \\
	&\hskip -2cm
	:= \sum_{(j_{l,q})} [q(p_{(1)}(e_{j_{1,1}}, \dots, e_{j_{r,1}}),\dots,p_{(m)}(e_{j_{1,m}},\dots, e_{j_{r,m}}))],
\end{align*}
where $(j_{l,q})$ ranges over the same set as in Lemma~\ref{lem:Mackey}.
\end{corollary}

\begin{corollary}[Splitting formula]
	For $f_1,\dots,f_n,g \in \Sigma^{\propP_\Omega(E)}$ and $p \in \propP(n)$,
	\[
	p(f_1\otimes\cdots\otimes f_n)\circ g
	= \sum p(f_1\circ g_{(1)} \otimes \cdots \otimes f_n \circ g_{(n)}).
	\]
\end{corollary}

\begin{proof}
	This follows from the compatibility $\Delta(f(u)) = \Delta(f)\Delta(u)$.
\end{proof}

\subsection{A twisted analogue of $\propP_\Omega(E)$  and Solomon's descent algebra} 
Following the approach of \cite{MR2074989}, it is expected that a combinatorial realization of $\propP_\Omega(E)$ in $\Smod$ would provide a broader setting for the Malvenuto--Reutenauer algebra and, consequently, for Solomon's descent algebra.

To define a twisted version of $\propP_{\Omega}(E)$, we consider the non-counital coalgebra $E'^\tau := \bigoplus_{n \geq 1} \field e_n$ in $\Smod$, equipped with coproduct
\begin{displaymath}
	\Delta'^\tau(e_n) := \sum_{i=1}^{n-1} e_i \otimes e_{n-i} \uotimes q_{i,n-i},
\end{displaymath}
where $q_{i,n-i}$ denotes the sum of all $(i,n-i)$-shuffles in $\bS_n$. This extends to a counital coalgebra $E^\tau := \field 1_0 \oplus E'^\tau$ with coproduct $\Delta^\tau$.

In analogy with \eqref{eq:C_relatively_free}, we define
\begin{displaymath}
	\propP_{\Omega}(E^\tau) := \propP \circ E'^{\tau} \big/ \tilde{\Omega}(E^\tau),
\end{displaymath}
which inherits the structure of a Hopf $\propP$-algebra in $\Smod$.

Its symmetrization $\overline{\propP_\Omega(E^\tau)}$ is then a Hopf $\propP$-algebra in $\grVect$ \cite{MR2674654}. Since $\propP_\Omega(E^\tau)$ is an algebra over the Hopf operad $\propP_\Omega$, we obtain:

\begin{corollary}\label{cor:Livernet}
	$\overline{\propP_{\Omega}(E^\tau)}$ is a Hopf $\propP_\Omega$-algebra in $\grVect$. In particular, the evaluation of $\tilde{\Omega}$ on this algebra is trivial.
\end{corollary}

\begin{remark}\label{cor:symmetrized_is_Omega_algebra}
	In general, $\overline{\propP_\Omega(E^\tau)}$ is not cocommutative (see, for example, $\field[\comperm]$ in Section~\ref{sec:absolutely_free}).
\end{remark}

We now recall the underlying vector space of Solomon's descent algebra. A descent of a permutation $\sigma = (\sigma(1), \dots, \sigma(n))$ is an index $1 \leq i \leq n-1$ such that $\sigma(i) > \sigma(i+1)$. For each pseudo-composition $I=(i_1,\dots, i_r)$ of $n$, let $D_{\leq I}$ be the sum of all permutations whose descent set is contained in $\{i_1, i_1+i_2, \dots, i_1+\cdots+i_{r-1}\}$. These elements span a subspace $\Sol(\bS)(n)$.

The graded space $\Sol(\bS) :=\bigoplus_{n\ge0}\Sol(\bS)(n)$ is the underlying space of Solomon's descent algebra, which is a Hopf subalgebra of $\field[\bS] :=\bigoplus_{n\ge0}\field[\bS_n]$ under the Malvenuto--Reutenauer structure. Each component is stable under the internal product $\circ$ induced by the group algebra, extended by $0$ across different degrees.

Motivated by this situation, we define, for $p \in \propP(r)$ and a pseudo-composition $I=(i_1,\dots,i_r)$ of $n$,
\begin{displaymath}
	D^p_{\leq I} := \left[\sum_{\typ(\pi)=I} p \uotimes e_{i_1}\otimes\cdots\otimes e_{i_r} \uotimes \Sh(\pi)\right] \in \propP_\Omega(E^\tau)(n).
\end{displaymath}

These elements span subspaces $\sol_\Omega(n)\subseteq \propP_\Omega(E^\tau)(n)$, and we set
\begin{displaymath}
	\sol_\Omega :=\bigoplus_{n\ge0}\sol_\Omega(n)\subseteq \propP_\Omega(E^\tau).
\end{displaymath}

\begin{remark}
	For a given $I$, the tensor $e_{i_1}\otimes\cdots\otimes e_{i_r}\uotimes\Sh(\pi)$ is determined by $\Sh(\pi)^{-1}$, which corresponds to a permutation with descents contained in $\{i_1,\dots,i_1+\cdots+i_r\}$.
\end{remark}

\begin{theorem}\label{thm:isomorphism_symmetric_Solomon}
	The space $\sol_\Omega$ is a Hopf $\propP$-subalgebra of $\overline{\propP_\Omega(E^\tau)}$. Moreover,
	\[
	\psi\colon \propP_\Omega(E)\longrightarrow \sol_\Omega,\qquad
	[p\uotimes e_{i_1}\otimes\cdots\otimes e_{i_m}]\mapsto D^p_{\leq(i_1,\dots,i_m)}
	\]
	is an isomorphism of Hopf $\propP$-algebras in $\grVect$.
\end{theorem}

\begin{proof}
	Let $\hat{p}$ denote the product induced in $\overline{\propP_\Omega(E^\tau)}$, and $p$ that of $\propP_\Omega(E^\tau)$. For any composition $(i_1,\dots,i_r)\models n$,
	\begin{align*}
		\hat{p}(D^{1_1}_{\le(i_1)},\dots,D^{1_1}_{\le(i_r)})
		&= \hat{p}([e_{i_1}],\dots,[e_{i_r}]) = [p(e_{i_1}\otimes\cdots\otimes e_{i_r}\uotimes q_{i_1,\dots,i_r})] \\
		&= D^p_{\le(i_1,\dots,i_r)}.
	\end{align*}
	This extends to arbitrary products:
	\[
	\hat{p}(D^{p_1}_{\le I_1},\dots,D^{p_r}_{\le I_r})
	= D^{p(p_1,\dots,p_r)}_{\le I_1\Vert\cdots\Vert I_r},
	\]
	where $\rVert$ denotes concatenation, showing that $\sol_\Omega$ is the $\propP$-subalgebra generated by the elements $D^{1_1}_{\le(i)}$. For the coproduct, we have
	\begin{equation}\label{eq:Delta_hat_clean}
		\hat{\Delta}(D^{1_1}_{\le(n)})=\sum_{i=0}^n D^{1_1}_{\le(i)}\otimes D^{1_1}_{\le(n-i)}.
	\end{equation}
	
	The homomorphism  $\psi_0:\propP\circ E'\longrightarrow \sol_\Omega$ determined by is an epimorphisms of Hopf $\propP$-algebras. Thus, it only remains to show that $\ker\psi_0=\tilde{\Omega}(E)$.
	
By Corollary~\ref{cor:Livernet}, we have $\tilde{\Omega}(E)\subseteq\ker\psi_0$.  Conversely, the space $\propP \circ\, E'$ admits a decomposition as a direct sum of homogeneous components, where the component corresponding to a partition $\{i_1,\dots,i_r\}$ of $n$ is 
\[
\propP(r) \uotimes e_{i_1} \otimes \cdots \otimes e_{i_r}.
\]
Similarly, for $\propP \circ\, E'^{\tau}$, the corresponding homogeneous component is 
\[
\propP(r) \uotimes e_{i_1} \otimes \cdots \otimes e_{i_r} \uotimes \Sh_{i_1,\dots,i_r}.
\]

Since $\tilde{\Omega}(E^\tau)$ is a homogeneous subspace of $\propP \circ\, E'^\tau$, it follows that $\ker \psi_0$ is itself a homogeneous subspace of $\propP \circ\, E'$. Moreover, there is a well-defined projection
\begin{align*}
	\propP \circ\, E'^{\tau} &\longrightarrow \propP \circ\, E', \\
	p \uotimes e_{i_1} \otimes \cdots \otimes e_{i_r} \uotimes \sigma 
	&\longmapsto 
	p \uotimes e_{i_1} \otimes \cdots \otimes e_{i_r},
\end{align*}
which is well-defined because each $\field e_m$ is a trivial $\bS$-module. Now let $p \uotimes e_{i_1} \otimes \cdots \otimes e_{i_r} \in \ker \psi_0$. Then the element
\[
p \uotimes e_{i_1} \otimes \cdots \otimes e_{i_r} \uotimes q_{i_1,\dots,i_r}
\]
can be written as
\[
\sum_j h_j \uotimes e_{i_1} \otimes \cdots \otimes e_{i_r} \uotimes \sigma_j,
\]
with $h_j \in \tilde{\Omega}$ and $\sigma_j \in \bS_n$. Applying the projection yields
\[
\frac{n!}{i_1!\cdots i_r!}\, 
p \uotimes e_{i_1} \otimes \cdots \otimes e_{i_r}
=
\sum_j h_j \uotimes e_{i_1} \otimes \cdots \otimes e_{i_r}.
\]

This shows that $p \uotimes e_{i_1} \otimes \cdots \otimes e_{i_r} \in \tilde{\Omega}(E)$, hence $\ker \psi_0 \subseteq \tilde{\Omega}(E)$, which completes the proof.
\end{proof}

\begin{remark}
	If $\operatorname{char}(\field)=p>0$, the map $\psi$ need not be injective, since
	\begin{displaymath}
	\psi([p\otimes e_1^{\otimes m}])=m!\,p\otimes e_1^{\otimes m}\uotimes\Id.
	\end{displaymath}
\end{remark}

\begin{corollary}\label{cor:main}
	We have the following isomorphisms of Hopf $\propP$-algebras in $\grVect$:
	\begin{displaymath}
		\Sigma^{\propP_\Omega(E)} \cong \propP_\Omega(E) \cong \sol_\Omega.
	\end{displaymath}
\end{corollary}

We now describe the internal product on $\sol_\Omega$. To this end, we consider the set $\comp^\#_n$ of pseudo-compositions of $\{1,\dots,n\}$. Any $\pi=(\pi_1,\dots,\pi_r) \in \comp^\#_n$ uniquely determines a permutation $\sigma^\pi=\Sh(\pi)^{-1}$ obtained by listing the elements of $\pi_1$ increasingly, followed by those of $\pi_2$, and so on. For example, if
\[
\pi = (\{3,4\}, \emptyset, \{5,2\}, \{1\}),
\]
then $\sigma^\pi = (3,4,2,5,1)$.

The symmetric group $\bS_n$ acts on $\comp^\#_n$ by
\[
\sigma(\pi) = (\sigma(\pi_1),\dots,\sigma(\pi_r)).
\]
Given $\pi,\pi'\in\comp^\#_n$, we define
\[
\pi(\pi') := \sigma^\pi(\pi'),
\]
and introduce the product
\begin{equation}\label{eq:circ_clean}
	\pi \circ \pi' :=
	\bigl(
	\pi(\pi'_1)\cap \pi_1,\dots,\pi(\pi'_1)\cap \pi_r,\;
	\dots,\;
	\pi(\pi'_s)\cap \pi_1,\dots,\pi(\pi'_s)\cap \pi_r
	\bigr),
\end{equation}
where $\pi'=(\pi'_1,\dots,\pi'_s)$. In other words, for each block $\pi'_k$ we distribute $\pi(\pi'_k)$ across the blocks of $\pi$ and list the intersections in lexicographic order (first varying $k$, then the components of $\pi$). Empty parts may be omitted, yielding a well-defined product on compositions.

This product admits the interpretation
\[
\Sh(\pi)^{-1}\Sh(\pi')^{-1} = \Sh(\pi \circ \pi')^{-1}.
\]

\begin{example}
	For
	\[
	\pi =(\{3,5,7\},\{1,4\},\{2,6\}), \qquad
	\bar{\pi} = (\{4,6,7\}, \{1,2,3,5\}),
	\]
	we obtain
	\[
	\pi(\bar{\pi}) =(\{1,2,6\},\{3,4,5,7\}),
	\]
	and
	\[
	\pi \circ \bar{\pi} = (\emptyset,\{1\},\{2,6\},\{3,5,7\},\{4\},\emptyset),
	\]
	which satisfies
	\[
	\Sh(\pi)^{-1} \Sh(\bar{\pi})^{-1}
	= (3,5,7,1,4,2,6)(4,6,7,1,2,3,5)
	= (1,2,6,3,5,7,4)
	= \Sh(\pi \circ \bar{\pi})^{-1}.
	\]
\end{example}

We denote the opposite product by $\pi \bullet \pi' := \pi' \circ \pi$.

\begin{proposition}
	The pair $(\comp^\#_n,\circ)$ is a monoid with identity $(\{1,\dots,n\})$. Moreover, the map
	\[
	\comp^\#_n \longrightarrow \bS_n, \qquad
	\pi \mapsto \sigma^\pi,
	\]
	is an epimorphism of monoids.
\end{proposition}

We now describe the induced product on $\sol_\Omega$. For $p \in \propP(r)$ and $\pi \in \comp^\#_n$ of type $(i_1,\dots,i_r)$, we write
\[
p\pi = p \uotimes e_{i_1} \otimes \cdots \otimes e_{i_r} \uotimes \Sh(\pi).
\]

For $\tau \in \bS_r$, these satisfy the equivariance relation
\[
(p\tau)\pi = p(\tau \pi),
\qquad
\tau \pi = (\pi_{\tau^{-1}(1)}, \dots, \pi_{\tau^{-1}(r)}).
\]

To define the product, we first work in the larger space
\[
\bigoplus_{m \geq 0} \prop{P}(m) \otimes \field[\comp^\#_m].
\]
For elements $p \in \propP(r)$ and $p' \in \propP(s)$ we set
\[
(p \otimes \pi) \bullet (p' \otimes \pi')
= p(p'_{(1)}, \dots, p'_{(r)}) \otimes (\pi \bullet \pi').
\]

This product is compatible with the equivariance in the first factor, and therefore descends to expressions of the form $p\pi$.

However, equivariance in the second factor does not hold in general. More precisely, writing
\[
\pi \wedge \pi' :=
(\pi_1 \cap \pi'_1,\dots,\pi_r \cap \pi'_1,\dots,\pi_r \cap \pi'_s),
\]
one obtains
\begin{align*}
	(p \pi) \bullet (p' \tau \otimes \pi') &= \sum p(p'_{(1)}, \dots, p'_{(r)})\,\pi'(\pi) \wedge \tau \pi', \\
	(p \pi) \bullet (p' \otimes \tau \pi') &= \sum p(p'_{(1)},\dots, p'_{(r)})\,(\tau \pi')(\pi) \wedge (\tau\pi').
\end{align*}
Since these expressions need not coincide, the product $\bullet$ does not extend to the whole space $\propP \circ E'^\tau$.

Nevertheless, for the elements $D^p_{\leq I}$ the dependence on representatives cancels, and the product becomes well-defined:
\[
D^p_{\leq I} \bullet p'\pi'
= \sum_{\typ(\pi)=I} p(p'_{(1)},\dots, p'_{(r)})\, \pi \bullet \pi'.
\]
It follows that for $I' = (k_1,\dots, k_m)$,
\begin{equation}\label{eq:product_DI}
D^p_{\leq I} \bullet D^{p'}_{\leq I'} 
= \sum_J D^{p(p'_{(1)},\dots, p'_{(r)})}_{\leq J},
\end{equation}
where $J$ is determined by the same combinatorics as in Lemma~\ref{lem:Mackey}.

\begin{corollary}\label{cor:internal_product_on_P-Sol}
	Let $\psi$ and $\varphi$ be the isomorphisms from Theorems~\ref{thm:isomorphism_symmetric_Solomon} and \ref{thm:isomorphism_symmetric_convolution}, respectively. Then $\psi\varphi^{-1}$ is an isomorphism between the algebras $(\Sigma^{\propP_\Omega(E)}, \circ)$ and $(\sol_\Omega, \bullet)$. 
\end{corollary}

\begin{remark}
	While it is tempting to view $\overline{\propP_\Omega(E^\tau)}$ as a generalization of the Malvenuto--Reutenauer Hopf algebra, the absence of a globally well-defined product $\bullet$ precludes such a direct identification. 
\end{remark}

These constructions will be made explicit in the case of universal enveloping algebras in the next section.

\section{Example: Universal enveloping algebras of Sabinin algebras}
\label{sec:example}
Sabinin algebras constitute the infinitesimal counterparts of analytic loops \cite{MR924255}, representing 
a non-associative generalization of Lie algebras that encompasses structures such as 
Malcev algebras, Bol algebras, and Lie triple systems. These algebras can be 
characterized as the spaces of primitive elements of suitable universal enveloping 
algebras—essentially, a class of cocommutative but non-associative analogues of Hopf 
algebras \cite{MR3174282}. 

In \cite{MR2504663}, Loday successfully undertook a vast generalization of the 
classical ``Three Graces'' $\propCom$-$\propAssoc$-$\propLie$. He proved that, given 
a type of generalized bialgebras $(\propC^c, \propP)$ over a field of characteristic 
zero, if it satisfies the following conditions:
\begin{enumerate}
	\item[(H0)] For any pair $(\delta, \mu)$ of a cooperation $\delta$ and an operation $\mu$, there is a distributive compatibility relation.
	\item[(H1)] The free $\propP$-algebra $\propP \circ A$ is naturally equipped with a $\propC^c\text{-}\propP$-bialgebra structure.
	\item[(H2epi)] The natural coalgebra map $\varphi(V) \colon \propP \circ V \longrightarrow \propC^c \circ V$ is surjective and admits a natural coalgebra splitting $s(V) \colon \propC^c \circ V \longrightarrow \propP \circ V$.
\end{enumerate}
Then, for any $\propC^c\text{-}\propP$-bialgebra $H$, the following assertions are equivalent:
\begin{enumerate}
	\item[a)] The $\propC^c\text{-}\propP$-bialgebra $H$ is connected.
	\item[b)] There is an isomorphism of bialgebras $H \cong U(\Prim H)$.
	\item[c)] There is an isomorphism of connected coalgebras $H \cong \propC^c \circ \Prim H$.
\end{enumerate}

Loday first observed that, for types satisfying conditions (H0) and (H1), the operad $\propP$ contains a suboperad $\Prim \propP$ with the property that, for any $\propC^c\text{-}\propP$-bialgebra $H$, the space $\Prim H$ naturally carries the structure of a $\Prim \propP$-algebra. If, in addition, condition (H2epi) holds, then the triple $(\propC, \propP, \Prim \propP)$ forms a good triple of operads. The assumption (H2epi) allows for the construction of a functorial versal idempotent $e = e_H \colon H \longrightarrow H$. This idempotent is obtained by defining, for each $n \ge 2$, a map $\omega^{[n]} \colon H \longrightarrow H$ as the composite
\begin{displaymath}
	\omega^{[n]} \colon H \xrightarrow{\theta_n} \propC^c(n) \uotimes H^{\otimes n} \xrightarrow{s(n) \otimes \Id} \propP(n) \uotimes H^{\otimes n} \xrightarrow{\gamma_n} H,
\end{displaymath}
where $\theta_n$ and $\gamma_n$ are the corresponding structure maps. The versal idempotent is then defined as the infinite product
\begin{displaymath}
	e := (\Id - \omega^{[2]})(\Id - \omega^{[3]}) \cdots (\Id - \omega^{[n]}) \cdots = \Id + \sum_{r= 1}^\infty \sum_{2 \leq i_1 < \cdots < i_r}  (-1)^r \omega^{[i_1]} \cdots \omega^{[i_r]}.
\end{displaymath}
This idempotent projects $H$ onto $\Prim H$ and constitutes the main tool in Loday's result. As a byproduct, Loday also proved that $\Prim (\propP \circ V) = (\Prim \propP) \circ V$.

Consider now a cocommutative connected Hopf symmetric operad $\propP$. The type $\propCom^c\text{-}\propP$ automatically satisfies (H0) and (H1). Assuming further that $\propCom^c\text{-}\propP$ satisfies (H2epi), we note that $\propP \circ E'$ is a connected $\propCom^c\text{-}\propP$-bialgebra; hence, $\propP \circ E' \cong U(\Prim(\propP \circ E'))$. Let $c_n := e(e_n)$ for $n \geq 1$. From the definition of $e$, the $\propP$-subalgebras generated by $\{e_i\}_{i=0}^n$ and $\{c_i\}_{i=0}^n$ coincide (with $c_0 = e_0 = 1_0$), while $c_{n+1}$ remains outside the subalgebra generated by the previous terms. Thus, $C' = \bigoplus_{n \geq 1} \field c_n$ is spanned by primitive elements, leading to:
\begin{equation}
	\propP \circ E' = \propP \circ C' \cong U(\Prim \propP \circ C').
\end{equation}
Given $\tilde{\Omega}$ as in Proposition~\ref{prop:Omega}, let
\begin{displaymath}
	\Prim \propP_{\partial \Omega}(C') := (\Prim \propP \circ C') / \ideal_{\Prim \propP \circ C'}\langle e(\Omega (\propP \circ C'))\rangle.
\end{displaymath}
A standard argument via the coradical filtration shows that $\tilde{\Omega}(C)$ is the $\propP$-ideal generated by $e(\Omega(\propP \circ C'))$. Consequently:
\begin{displaymath}
	\propP_\Omega (C') = (\propP \circ C') / \tilde{\Omega}(C) \cong U(\Prim \propP_{\partial \Omega}(C') ).
\end{displaymath}

\begin{corollary}\label{cor:four-isomorphisms}
	Let $\propP$ be a cocommutative connected Hopf symmetric operad. If the type of generalized bialgebras $\propCom^c\text{-}\propP$ satisfies (H2epi), then:
	\begin{displaymath}
	\Sigma^{\propP_\Omega(E)} \cong \propP_\Omega(E) \cong \sol_\Omega \cong U(\Prim \propP_{\partial \Omega}(C') ).
	\end{displaymath}
\end{corollary}

Now we focus on our motivating example: the universal enveloping algebras of Sabinin algebras. A loop is a nonempty set equipped with a unit element and a binary product such that the left and right multiplication operators are bijections; intuitively, it can be viewed as a non-associative group. More precisely, a loop is a set with three operations $\cdot, \backslash, /$ subject to the following relations:
\begin{equation}\label{eq:loops}
	x \backslash (x \cdot y) = y = x \cdot (x \backslash y), \quad (x \cdot y)/y = x = (x/y) \cdot y \quad \text{and} \quad x \backslash x = y / y.
\end{equation}
The unit element is given by $e = x \backslash x$. In the category of pointed sets (with the Cartesian product as the tensor product), and identifying $\{e \} \times X \cong X \cong X \times \{e\}$, we define $\delta_2(x) = (x,x), \delta_1(x) = 1_1(x) = x$ and $\delta_0(x) = e$. Thus, the relations in \eqref{eq:loops} can be reformulated as:
\begin{gather*}
	\backslash (1_1 \otimes m)(\delta_2 \otimes 1_1) = \delta_0 \otimes 1_1, \quad m(1_1 \otimes \backslash)(\delta_2 \otimes 1_1) = \delta_0 \otimes 1_1, \\
	/(m \otimes 1_1)(1_1 \otimes \delta_2) = 1_1 \otimes \delta_0, \quad m(/\otimes 1_1) (1_1 \otimes \delta_2) = 1_1 \otimes \delta_0 \quad \text{and}\\
	(\backslash \otimes \delta_0)(\delta_2 \otimes 1_1) = (\delta_0 \otimes /)(1_1 \otimes \delta_2).
\end{gather*}
where $m$ stands for the produce $\cdot$. Furthermore, if we define
\begin{displaymath}
	\delta(p) := p \otimes p \quad \text{and} \quad \epsilon(p) := 1
\end{displaymath}
for any operation $p$ composed of $m, \backslash, /$, then relations (R1)--(R5) in Section~\ref{sec:relatively free} hold tautologically for loops.

To treat these as bialgebras, we consider the (linear) symmetric Hopf operad $\TMag$ generated by the three binary operations $m, \backslash$ and $/$. The Hopf structure arises from the linear extension of the aforementioned $\delta$ and $\epsilon$ to $\TMag$. We then consider the prop generated by $\TMag, \delta_2, \delta_1, \delta_0$ subject to (R1)--(R5). Since we aim to study ``loop objects'' in the category of coalgebras, we must either impose the reformulation of \eqref{eq:loops} to obtain an adequate prop $\propLoop$, or assume that the set $\Omega$ in our construction always contains
\begin{align*}
	\Omega_L := \{ &\backslash(1_1 \otimes m)(\delta_2 \otimes 1_1) - \delta_0 \otimes 1_1, m(1_1 \otimes \backslash)(\delta_2 \otimes 1_1) - \delta_0 \otimes 1_1, \\
	&/(m \otimes 1_1)(1_1 \otimes \delta_2) - 1_1 \otimes \delta_0, m(/\otimes 1_1)(1_1 \otimes \delta_2) - 1_1 \otimes \delta_0, \\
	&(\backslash \otimes \delta_0)(\delta_2 \otimes 1_1) - (\delta_0 \otimes /)(1_1 \otimes \delta_2) \}.
\end{align*}

It can be shown that $\Mag \circ E'$ is indeed a $\propLoop$-algebra, as it admits left and right divisions $\backslash, /$ defined recursively. For instance, $(\sum u/ e_{1(1)}) \cdot e_{1(2)} = \epsilon(e_1)$ implies $u / e_1 + u \cdot e_1 = 0$, i.e., $u/e_1 = - u \cdot e_1$. Consequently:

\begin{proposition}
	$\TMag \circ E' / \tilde{\Omega}_L(E) \cong \Mag \circ E'$ as $\propLoop$-algebras.
\end{proposition}

In short, the formalism of $\TMag$ can be bypassed by working directly with the $\propLoop$-algebra $\Mag \circ E'$, where $\backslash$ and $/$ are defined such that the the identities determined by $\Omega_L$ hold.

Since $(\propCom, \propMag, \propSab)$ constitutes a good triple of operads, Corollary~\ref{cor:four-isomorphisms} is directly applicable. In this context, the identities of interest are those naturally associated with loops; for instance, \eqref{eq:Moufang-Hopf} stems from the Moufang identity $x(y(xz)) = ((x y)x)z$. As established, any such identity can be expressed within the category of pointed sets in terms of $\cdot, \backslash, /, \delta_0, \delta_1$, and $\delta_2$. When viewed in $\propLoop$, these identities span a coideal $\Omega$ for which $\propMag_\Omega(E) = \propMag \circ E' /\tilde{\Omega}(E) \cong U(\propSab_{\partial \Omega}(C'))$. Consequently, for the universal enveloping algebra of the relatively free Sabinin algebra $\propSab_{\partial \Omega}\{ \bX \}$ on $\bX :=\{ x_1, x_2, \dots \}$, we have:
\begin{corollary}\label{cor:main_Sabinin} 
	\begin{displaymath}
		\Sigma^{\propMag_\Omega(E)} \cong \propMag_\Omega(E) \cong \magSol_\Omega \cong U(\propSab_{\partial \Omega}\{ \bX \})
	\end{displaymath}
\end{corollary}
This provides a structural explanation for the associative internal product appearing in the universal enveloping algebras of relatively free Sabinin algebras. A partial version of this result, obtained by different methods using bases of logarithms and avoiding an appeal to Loday’s theorem, can be found in the preprint \cite{1812.04450}, which presents an earlier and more restrictive approach that has not been developed further to date. The preprint \cite{1812.04450} may also be of interest to readers unfamiliar with the operadic language.

\section{The absolutely free non-associative algebra}
\label{sec:absolutely_free}
Although our internal associative product for $\sol_\Omega$ does not possess a canonical extension to a well-defined product on $\overline{\propP_\Omega(E^\tau)}$ in full generality, the situation becomes significantly more tractable for regular operads. We focus our discussion on the specific case where $\propP = \Mag$ and $\Omega = \{ 0 \}$---notwithstanding that other regular operads may be treated similarly---as this setting serves as the primary motivation for this work. Since $\Mag$ is a regular operad, for each $n \geq 1$, the space $\overline{\Mag(E^\tau)}(n)$ admits a $\field$-basis $\comperm_n$ consisting of elements of the form
\begin{displaymath}
	t\pi := t \otimes e_{i_1} \otimes \cdots \otimes e_{i_r} \otimes \Sh(\pi),
\end{displaymath}
where $I = (i_1, \dots, i_r)$ ranges over the compositions of $n$, $\pi$ ranges over the compositions of $\conj{n}$ of type $I$, and $t$ belongs to the set $\PT_n$ of binary rooted planar trees with $n$ leaves. For convenience, we may relax this notation by allowing $I$ and $\pi$ to be pseudo-compositions, subject to the canonical identification.

\begin{example}
	\begin{displaymath}
		\gbeg{7}{5}
		\gmu				\linea   \gvac{1}   \gmu 				\linea   	\gnl
		\gcn{2}{1}{2}{3} 	\linea 	  \gvac{1}	 \gcn{2}{1}{2}{3} 	\linea		\gnl
		\gvac{1} 			\gmu      \gvac{2} 	\gmu							\gnl
		\gcn{2}{1}{4}{7}			  			\gcn{2}{1}{8}{5} 				\gnl
		\gvac{3} \gmu
		\gend
		(12, \emptyset, \emptyset, 34, 5, 6) 
		=
		\gbeg{4}{5}
		\linea \gmu \linea \gnl
		\linea \gcn{2}{1}{2}{3} \linea \gnl
		\linea \gvac{1} \gmu  \gnl
		\linea \gcn{1}{1}{4}{1} \gnl
		\gmu
		\gend
		(12, 34, 5, 6)
	\end{displaymath}
	where, to simplify the notation in examples, we denote $(\{ 1,2 \}, \emptyset, \emptyset, \{ 3,4 \}, \{ 5 \}, \{ 6 \})$ by $(12, \emptyset, \emptyset, 34, 5, 6)$.
\end{example}

For $n = 0$, we define $\comperm_0 := 1_0$. Furthermore, for any $n \geq 0$ and $t \in \PT_n$, we set $t(\emptyset, \dots, \emptyset) := 1_0$.

\begin{definition}
	The elements of $\comperm := \bigcup_{n\geq 0} \comperm_n$ are referred to as compermutations.
\end{definition}

Compermutations constitute a basis for 
\begin{displaymath}
	\field[\comperm] := \overline{\Mag(E^\tau)}.
\end{displaymath}
Letting $\field[\comperm_n]$ denote the linear span of $\comperm_n$, we have the decomposition $\field[\comperm] = \bigoplus_{n\geq 0} \field[\comperm_n]$. A natural product between two compermutations $t\pi$ and $t'\pi'$ is defined by
\begin{displaymath}
	t\pi \utimes t'\pi' := (t \vee t') (\pi \utimes \pi'),
\end{displaymath}
where $t \vee t'$ denotes the join of $t$ and $t'$.

\begin{example}
	\begin{displaymath}
		\gbeg{2}{1}
		\gmu
		\gend
		(1,2) \utimes 
		\gbeg{2}{1}
		\gmu
		\gend
		(2,13) = 
		\gbeg{4}{3}
		\gmu  \gmu \gnl
		\gcn{1}{1}{2}{3}  \gcn{1}{1}{4}{3} \gnl
		\gvac{1}\gmu
		\gend (1,2,4,35).
	\end{displaymath}
\end{example}

Of greater interest is the coassociative unital bialgebra structure $(\field[\comperm], *, \Delta)$ obtained via the symmetrization $\field[\comperm] = \overline{\Mag(E^\tau)}$. For basic elements $t\pi \in \field[\comperm_m]$ and $t'\pi' \in \field[\comperm_n]$, the product is given by
\begin{equation}\label{eq:multiplication_comperm}
	{t\pi} * {t'\pi'} := \sum_{\gamma \in \Sh_{m,n}} {t \vee t' \gamma (\pi \utimes \pi')} \in \field[\comperm_{m+n}],
\end{equation}
where $\Sh_{m,n}$ denotes the set of $(m,n)$-shuffles in $\bS_{m+n}$. The comultiplication is defined as
\begin{equation}\label{eq:comultiplication_comperm}
	\Delta(t\pi) := \sum_{p= 0}^m t\pi\vert_{\conj{p}} \otimes t \st(\pi\vert_{\conj{m} \setminus \conj{p}}) \in 
	\bigoplus_{p = 0}^m \field[\comperm_p] \otimes \field[\comperm_{m-p}],
\end{equation}
noting that $t\pi\vert_{\conj{0}} = 1_0$. Additionally, an internal product is defined by 
\begin{displaymath}
	t\pi \circ t'\pi' := \left\{ \begin{array}{ll} t'(t,\dots,t) \pi \circ \pi' & \text{if } m = n, \\ 0 & \text{if } m \neq n, \end{array}\right.
\end{displaymath}
where $\pi \circ \pi'$ is the product in $\comp^\#_n$ defined in \eqref{eq:circ_clean}.

\begin{example}
	\begin{align*}
		\Delta\left(
		\gbeg{3}{3}
		\gmu  \linea \gnl 
		\gcn{1}{1}{2}{3} \gvac{1} \linea \gnl
		\gvac{1} \gmu
		\gend		
		(1,2,3) 
		\right)
		&= 
		1_0 \otimes 	\gbeg{3}{3}
		\gmu  \linea \gnl 
		\gcn{1}{1}{2}{3} \gvac{1} \linea \gnl
		\gvac{1} \gmu
		\gend		
		(1,2,3) 
		+
		\gbeg{1}{1}
		\linea 
		\gend
		(1) 
		\otimes 
		\gbeg{2}{1}
		\gmu
		\gend
		(1,2) 
		+
		\gbeg{2}{1}
		\gmu
		\gend
		(1,2) 
		\otimes
		\gbeg{1}{1}
		\linea 
		\gend
		(1) \\
		& + 
		\gbeg{3}{3}
		\gmu  \linea \gnl 
		\gcn{1}{1}{2}{3} \gvac{1} \linea \gnl
		\gvac{1} \gmu
		\gend		
		(1,2,3) \otimes 1_0.
	\end{align*}
\end{example}

\begin{example}
	\begin{displaymath}
		\gbeg{3}{3}
		\gmu  \linea \gnl 
		\gcn{1}{1}{2}{3} \gvac{1} \linea \gnl
		\gvac{1} \gmu
		\gend		
		(1,3,2) \circ 
		\gbeg{2}{1}
		\gmu
		\gend
		(1,23)
		=
		\gbeg{7}{5}
		\gmu				\linea   \gvac{1}   \gmu 				\linea   	\gnl
		\gcn{2}{1}{2}{3} 	\linea 	  \gvac{1}	 \gcn{2}{1}{2}{3} 	\linea		\gnl
		\gvac{1} 			\gmu      \gvac{2} 	\gmu							\gnl
		\gcn{2}{1}{4}{7}			  			\gcn{2}{1}{8}{5} 				\gnl
		\gvac{3} \gmu
		\gend
		(1, \emptyset, \emptyset, \emptyset, 3, 2) 
		= 
		\gbeg{3}{3}
		\linea \gmu \gnl
		\linea \gcn{1}{1}{2}{1} \gnl
		\gmu \gvac{1} 
		\gend
		(1,3,2)
	\end{displaymath}
\end{example}

The central hypothesis of this work—namely, that the absolutely free non-associative algebra $\field\{\bX\} = \prop\Mag \circ \bX$ admits a natural internal associative product, making it a non-associative analogue of Solomon's descent algebra within a non-associative Malvenuto--Reutenauer Hopf algebra—can now be formally established. Recall that the structure of the associative Malvenuto--Reutenauer Hopf algebra $(\field[\bS], *, \Delta, \circ)$ is induced by $\sigma * \sigma':= \sum_{\gamma \in \Sh_{m,n}} \gamma (\sigma \utimes \sigma')$, $\Delta(\sigma) := \sum_{p= 0}^m t\sigma\vert_{\conj{p}} \otimes t \st(\sigma\vert_{\conj{m} \setminus \conj{p}})$  and the composition $\circ$ of permutations.

\begin{proposition}
	The mapping
	\begin{align*}
		\zeta \colon (\field[\comperm], *, \Delta, \circ) &\longrightarrow (\field[\bS], *, \Delta, \circ) \\
		{t\pi} &\mapsto \sigma^{\pi}
	\end{align*}
	is an epimorphism of bialgebras and a homomorphism with respect to the internal product $\circ$.
\end{proposition}

\begin{proof}
	It follows directly that
	\begin{align*}
		\zeta(t\pi * t' \pi') &= \sum_{\gamma \in \Sh_{m,n}} \sigma^{\gamma(\pi \utimes \pi')} = \sum_{\gamma \in \Sh_{m,n}} \gamma(\sigma^\pi \utimes \sigma^{\pi'}) = \zeta(t\pi) * \zeta(t'\pi'),\\
		\Delta(\zeta(t\pi)) &= \sum_{p=0}^m \sigma^\pi\vert_{\conj{p}} \otimes \st(\sigma^\pi\vert_{\conj{m} \setminus \conj{p}}) = (\zeta \otimes \zeta) (\Delta(t\pi)), \quad \text{and} \\
		\zeta(t\pi \circ t'\pi') &= \sigma^{\pi \circ \pi'} = \sigma^\pi \circ \sigma^{\pi'} = \zeta(t\pi) \circ \zeta(t'\pi').
	\end{align*}
\end{proof}

As $I = (i_1,\dots, i_r) \models m$ and $t \in \PT_r$ vary over all possibilities, the elements 
\begin{displaymath}
	D^t_{\leq I} := \sum_{\typ(\pi) = I} t\pi \in \field[\comperm]
\end{displaymath}
span a subspace $\Sol(\comperm) \subseteq \field[\comperm]$ that is closed under $*$, $\Delta$, and $\bullet$, as verified by the following relations:
\begin{align*}
	D^t_{\leq I} * D^{t'}_{\leq I'} &= D^{t \vee t'}_{\leq I \parallel I'}, \quad \Delta(D^\vert_{\leq (m)}) = \sum_{i=0}^m D^\vert_{\leq (i)} \otimes D^\vert_{\leq (m-i)}, \quad \text{and}\\
	D^t_{\leq I} \bullet D^{t'}_{\leq I'} &= \sum_J D^{t(t',\dots, t')}_{\leq J},
\end{align*}
where $J$ ranges over the set  defined in Lemma~\ref{lem:Mackey}, and $\vert = 1_1$. Since $\propMag \circ \bX = \field\{ \bX\}$, Corollary~\ref{cor:main} and formula \eqref{eq:product_DI} give:

\begin{corollary}\label{eq:cor_main_compermutations}
	The following isomorphisms hold:
	\begin{displaymath}
		(\Sigma^{\field\{\bX\}}, *, \Delta) \cong (\field\{\bX\}, *, \Delta) \cong (\Sol(\comperm), *, \Delta).
	\end{displaymath}
	Furthermore, $(\Sol(\comperm), \bullet) \cong (\Sigma^{\field\{\bX\}}, \circ)$.
\end{corollary}

\begin{remark}
	The internal product for pseudo-compositions employed here differs from the product $\wedge$ considered, for instance, by Bidigare \cite{Bid97} (see also \cite{Br00} and \cite[Appendix B]{MR2424338}). For completeness, we recall Bidigare's approach to Solomon's descent algebra. The monoid $\comp_m$ of compositions of $\conj{m}$ is equipped with the product
	\begin{displaymath}
		\pi \wedge \pi' = (\pi_1 \cap \pi'_1, \dots, \pi_1 \cap \pi'_s, \pi_2 \cap \pi'_1, \dots, \pi_r \cap \pi'_s).
	\end{displaymath}
	Let $\sigma(\pi) = (\tau(\pi_1), \dots, \tau(\pi_r))$ denote the natural action of $\bS_m$ on $\comp_m$. The $\mathbb{Z}$-module $\mathbb{Z}[\comp_m^{(1,\dots,1)}]$ spanned by $\comp_m^{(1,\dots,1)} = \{ \pi \in \comp_m \mid \typ(\pi) = (1,\dots,1)\}$ is isomorphic to $\mathbb{Z}[\bS_m]$ via $\pi \mapsto \sigma^{\pi}$, and constitutes a two-sided ideal with respect to $\wedge$. Under this isomorphism, we have $\sigma^{\tau(\pi)} = \tau \sigma^{\pi}$. Since $\tau(\pi \wedge \pi') = \tau(\pi) \wedge \tau(\pi')$, the left multiplication by elements of $\mathcal{B}_m := \{ X \in \mathbb{Z}[\comp_m] \mid \tau(X) = X \text{ for all } \tau \in \bS_m\}$ commutes with the $\bS_m$-action. 
	
	In the isomorphism $\mathbb{Z}[\comp_m^{(1,\dots,1)}] \cong \mathbb{Z}[\bS_m]$, the centralizer of this action is the algebra generated by right multiplication operators by elements of $\bS_m$, which is isomorphic to $(\mathbb{Z}[\bS_m], \bullet)$. The image of $\sum_{\typ(\pi) = I} \pi \in \mathcal{B}_m$ in the centralizer algebra is the right multiplication operator by $\sum_{\typ(\pi) = I} \sigma^{\pi} = \sum_{\Des(\sigma) \leq I} \sigma$. Evaluating these operators at $(1,\dots, m) \in \bS_m$ yields $\sum_{\Des(\sigma) \leq I} \sigma$. Thus, Bidigare's algebra $\mathcal{B}_m$ is isomorphic to $\Sol(\bS_m)$, where coefficients are taken from $\mathbb{Z}$. 
	
	Notably, Bidigare's product satisfies $\pi \wedge \pi = \pi$, a property that the product $\bullet$ lacks. For any $\pi \in \comp_m$ and $\pi' \in \comp^{(1,\dots,1)}_m$, we have:
	\begin{displaymath}
		\tau(\pi \bullet \pi') = (\tau\sigma^{\pi'}(\pi_i) \cap \tau(\pi'_j))_{i,j} = (\sigma^{\tau(\pi')}(\pi_i) \cap \tau(\pi')_j)_{i,j} = \pi \bullet \tau(\pi'),
	\end{displaymath}
	which demonstrates that the left multiplication $\pi' \mapsto \pi \bullet \pi'$ corresponds to the right multiplication operator by $\pi \bullet (\{1\},\dots, \{m\}) = \sigma^{\pi}$. While this is sufficient to show via $\zeta$ that $\Sol(\bS_m)$ is closed under $\bullet$, it does not immediately establish that $\field[\comperm_m]$ is also closed under this product.
\end{remark}

We now turn to the structure of $\field[\comperm]$. The structure of the (associative) Malvenuto--Reutenauer Hopf algebra was determined in \cite{MR1334836} and subsequently studied in greater detail in \cite{MR2103213}.

\subsection{The structure of $(\field[\comperm],*)$}

Given a compermutation $t\pi \in \comperm_m$, if it can be decomposed as $t\pi = t_1 \Pi_1 \utimes t_2 \Pi_2$ for some $t_1 \Pi_1 \in \comperm_i$ and $t_2 \Pi_2 \in \comperm_{m-i}$ with $1 \leq i \leq m-1$, we say that $|\Pi_1|$ is the principal breaking point of $t\pi$. If $t\pi$ admits no principal breaking point, it is called atomic. More generally, if 
\begin{displaymath}
	t\pi = t'(t_1, \dots, t_i \vee t_{i+1}, \dots, t_l) \Pi_1 \utimes \dots \utimes \Pi_i \utimes \Pi_{i+1} \utimes \dots \utimes \Pi_l
\end{displaymath}
where $t' \in \PT_{l-1}$ and $t_j \in \PT_{\longi(\Pi_j)}$ for $j= 1, \dots, l$, we say that $|\Pi_1| + \dots + |\Pi_i|$ is a breaking point of $t\pi$.

\begin{example}
	The compermutation 
	$
	\gbeg{4}{3}
	\gmu  \gmu \gnl
	\gcn{1}{1}{2}{3}  \gcn{1}{1}{4}{3} \gnl
	\gvac{1}\gmu
	\gend (1,2,3,4)
	$ 
	possesses breaking points at $1$, $2$ (principal), and $3$. In contrast, 
	$
	\gbeg{4}{3}
	\gmu  \gmu \gnl
	\gcn{1}{1}{2}{3}  \gcn{1}{1}{4}{3} \gnl
	\gvac{1}\gmu
	\gend (1,4,2,3)
	$ 
	has no breaking points and is, therefore, atomic.
\end{example}

Visually, $t\pi$ can be interpreted as the tree $t$ with its leaves decorated by the components of $\pi$. We say a node of $t$ is compatible (with $\pi$) if every entry in the leaves reachable from its left child is strictly smaller than every entry in the leaves reachable from its right child. For instance, in the compermutation:
\begin{displaymath}
	\gbeg{7}{7}
	\gnl
	\got{1}{{\text{\tiny 1}}} \got{1}{{\text{\tiny 2}}} \got{1}{{\text{\tiny 34}}} \got{2}{\,\,\,{\text{\tiny 6}}} \got{1}{{\text{\tiny 5}}} \got{1}{{\text{\tiny 7}}} \gnl
	\gmu				\linea   \gvac{1}   \gmu 				\linea   	\gnl
	\gcn{2}{1}{2}{3} 	\linea 	  \gvac{1}	 \gcn{2}{1}{2}{3} 	\linea		\gnl
	\gvac{1} 			\gmuc      \gvac{2} 	\gmu							\gnl
	\gcn{2}{1}{4}{7}			  			\gcn{2}{1}{8}{5} 				\gnl
	\gvac{3} \gmu
	\gend
\end{displaymath}
the circled node is compatible because its left child leads to $\{1, 2\}$ and its right child leads to $\{3, 4\}$. In the following figure, all compatible nodes have been marked:
\begin{displaymath}
	\gbeg{7}{7}
	\gnl
	\got{1}{{\text{\tiny 1}}} \got{1}{{\text{\tiny 2}}} \got{1}{{\text{\tiny 34}}} \got{2}{\,\,\,{\text{\tiny 6}}} \got{1}{{\text{\tiny 5}}} \got{1}{{\text{\tiny 7}}} \gnl
	\gmuc				\linea   \gvac{1}   \gmu 				\linea   	\gnl
	\gcn{2}{1}{2}{3} 	\linea 	  \gvac{1}	 \gcn{2}{1}{2}{3} 	\linea		\gnl
	\gvac{1} 			\gmuc      \gvac{2} 	\gmuc							\gnl
	\gcn{2}{1}{4}{7}			  			\gcn{2}{1}{8}{5} 				\gnl
	\gvac{3} \gmuc
	\gend
\end{displaymath}
It is easy to observe that the breaking points of $t\pi$ are determined by those compatible nodes for which the entire path to the root also consists of compatible nodes. Thus, the breaking points in this case are $1, 2, 4$, and $6$.

\begin{example}
	In the compermutation
	$\gbeg{7}{7}
	\gnl
	\got{1}{{\text{\tiny 1}}} \got{1}{{\text{\tiny 2}}} \got{1}{{\text{\tiny 35}}} \got{2}{\,\,\,{\text{\tiny 4}}} \got{1}{{\text{\tiny 6}}} \got{1}{{\text{\tiny 7}}} \gnl
	\gmuc				\linea   \gvac{1}   \gmuc				\linea   	\gnl
	\gcn{2}{1}{2}{3} 	\linea 	  \gvac{1}	 \gcn{2}{1}{2}{3} 	\linea		\gnl
	\gvac{1} 			\gmuc      \gvac{2} 	\gmuc							\gnl
	\gcn{2}{1}{4}{7}			  			\gcn{2}{1}{8}{5} 				\gnl
	\gvac{3} \gmu
	\gend
	$ 
	we find several compatible nodes, yet no breaking points exist because the root itself is not compatible. Consequently, this compermutation is atomic.
\end{example}

Any $t\pi$ can be uniquely factorized as $t\pi = t'(t_1, \dots, t_l) \Pi_1 \utimes \dots \utimes \Pi_l$, where each $t_j \Pi_j$ is atomic. We define a partial order on $\comperm_n$ by declaring $t'\pi' < t\pi$ if the set of breaking points of $t'\pi'$ is strictly contained in the set of breaking points of $t\pi$. From the unique factorization above, it is clear that minimal elements under this order are necessarily atomic.

\begin{theorem}\label{thm:classification_*}
	$(\field[\comperm], *)$ is a non-associative unital algebra freely generated by the set $\{ t\pi \in \comperm \mid t\pi \text{ is atomic}\}$.
\end{theorem}

\begin{proof}
	Consider a monomial $w^*(t_1\Pi_1, \dots, t_l \Pi_l)$ in $(\field[\comperm], *)$, where $t_1 \Pi_1, \dots, t_l \Pi_l$ are atomic and $w \in \PT_l$ is an element of the non-symmetric magmatic operad. The operation $w^*$ is the $l$-ary operation induced by $w$ on the algebra $(\field[\comperm], *)$. 
	
	Let $t\pi = w(t_1, \dots, t_l) \gamma (\Pi_1 \utimes \dots \utimes \Pi_l)$, with $\gamma \in \Sh_{|\Pi_1|, \dots, |\Pi_l|}$, be a summand in the expansion of $w^*(t_1 \Pi_1, \dots, t_l \Pi_l)$, and let $p$ be a breaking point of $t\pi$. If $p \notin \{ |\Pi_1|, |\Pi_1| + |\Pi_2|, \dots, |\Pi_1| + \dots + |\Pi_{l-1}| \}$, let $i$ be the smallest index such that $p < |\Pi_1| + \dots + |\Pi_i|$. Then the node in $w(t_1, \dots, t_l)$ corresponding to this breaking point must belong to the sub-tree $t_i$. Since the path of compatible nodes from this node to the root of $t\pi$ must consist of compatible nodes within $t_i$ (including the root of $t_i$), this would imply that $t_i \Pi_i$ has a breaking point, contradicting the assumption that $t_i \Pi_i$ is atomic. 
	
	Therefore, all breaking points of $t\pi$ must belong to $\{ |\Pi_1|, |\Pi_1| + |\Pi_2|, \dots, |\Pi_1| + \dots + |\Pi_{l-1}| \}$. This implies that for any summand $t\pi$ in the expression of the product $w^*(t_1 \Pi_1, \dots, t_l \Pi_l)$, we have $t\pi \leq w(t_1, \dots, t_l) \Pi_1 \utimes \dots \utimes \Pi_l$. The result then follows by a standard inductive argument.
\end{proof}

\subsection{The structure of $(\field[\comperm], \Delta)$}

We consider $\field[\comperm]$ equipped with the dual operations $*'$ and $\Delta'$ of $*$ and $\Delta$, respectively, defined by
\begin{displaymath}
	t\pi *' t'\pi' := \sum_{\substack{t'' \pi''\vert_{\conj{m}} = t\pi \\t''\st(\pi''\vert_{\conj{m+n} \setminus \conj{m}}) = t'\pi'}} t'' \pi''
\end{displaymath}
and
\begin{displaymath}
	\Delta'(t\pi) := \sum_{t = t' \vee t''} t'\st(\pi_1,\dots, \pi_r) \otimes t'' \st(\pi_{r+1},\dots, \pi_{r+s}).
\end{displaymath}

For any compermutation $t\pi \in \comperm$, there exists a unique factorization $t\pi = t_l \Pi_l \utimes (\dots \utimes (t_2 \Pi_2 \utimes t_1 \Pi_1))$ with the maximum possible number of factors. We say that $t\pi$ is connected if this decomposition involves an odd number of factors; otherwise, we say that $t\pi$ is not connected. It is clear that either $t\pi$ is connected or it can be written as $t\pi = t' \pi' \utimes t''\pi''$ with $t'' \pi''$ connected. Consequently, we obtain a factorization $t\pi = ((t_1 \Pi_1 \utimes t_2 \Pi_2) \utimes \dots) \utimes t_l \Pi_l$ where each $t_i \Pi_i$ is connected. 

In fact, this decomposition is unique. Suppose we have two different factorizations into connected factors:
\begin{displaymath}
	t\pi = ((t_1 \Pi_1 \utimes t_2 \Pi_2) \utimes \dots)\utimes t_l \Pi_l = ((t'_1 \Pi'_1 \utimes t'_2 \Pi'_2) \utimes \dots) \utimes t'_{l'}\Pi'_{l'}.
\end{displaymath}
If the right-hand side involves only one factor, say $t'\pi'$, then $t\pi$ is connected. However, if $t_l \Pi_l$ is also connected, then $t\pi = ((t_1 \Pi_1 \utimes t_2 \Pi_2) \utimes \dots)\utimes t_l \Pi_l$ is not connected, which is a contradiction. Thus, both sides must involve at least two factors, implying $t_l \Pi_l = t'_{l'} \Pi'_{l'}$. By iterating this argument, we establish uniqueness. Based on this factorization into connected compermutations, we define the connected type of $t\pi$ as $\ctyp(t\pi) = (\vert \Pi_1 \vert, \dots, \vert \Pi_l \vert)$.

The following result relies on the uniqueness of the previous factorization and a suitable partial order. For $t\pi, t'\pi' \in \comperm_n$, we say that $t'\pi' < t\pi$ if:
\begin{itemize}
	\item[i)] the set of breaking points of $t'\pi'$ is strictly contained in that of $t\pi$, or
	\item[ii)] $t'\pi'$ and $t\pi$ share the same breaking points and, reading from right to left, the first component in which $\ctyp(t'\pi')$ and $\ctyp(t\pi)$ differ is greater in $\ctyp(t'\pi')$. 
\end{itemize}
While thinking in terms of compatible nodes is useful for the following proof, the reader should note that the compermutations $t''\pi''$ appearing in the expansion of $t\pi *' t'\pi'$ do not necessarily share the same trees or the same type of pseudo-compositions.

\begin{theorem}\label{thm:structure_dual}
	$(\field[\comperm], *')$ is a unital associative algebra freely generated by the set $\{t\pi \in \comperm \mid t \pi \text{ is connected}\}$. 
\end{theorem}

\begin{proof}
	We examine the summands in the expression $t_1 \Pi_1 *' t_2 \Pi_2 *' \dots *' t_l \Pi_l$ for connected factors. Our goal is to compare them with $((t_1 \Pi_1 \utimes t_2 \Pi_2) \utimes \dots) \utimes t_l \Pi_l$. Defining the sets $S_1= \{ 1,\dots, \vert \Pi_1 \vert \}, \dots, S_l = \{ \vert \Pi_1 \vert + \dots + \vert \Pi_{l-1} \vert + 1, \dots, n \}$, the summands under consideration are those $t\pi \in \comperm_n$ such that $t\st(\pi\vert_{S_j}) = t_j \Pi_j$ for $j=1,\dots, l$.
	
	Let $\pi = (\pi_1,\dots, \pi_r)$ and assume $a$ is a breaking point of $t\pi$ not belonging to $\{ \vert \Pi_1 \vert, \vert \Pi_1 \vert + \vert \Pi_2 \vert, \dots \}$. Let $j \geq 1$ be the minimum index such that $a < \vert \Pi_1 \vert + \dots + \vert \Pi_j \vert$. Then $a = \vert \Pi_1 \vert + \dots + \vert \Pi_{j-1} \vert + b$ and $\vert \Pi_j \vert = b + c$ with $b, c \geq 1$. Since $t\st(\pi\vert_{S_j}) = t_j \Pi_j$, it follows that $b$ is a breaking point of $t_j \Pi_j$ and $a$ is a breaking point of $((t_1 \Pi_1 \utimes t_2 \Pi_2) \utimes \dots) \utimes t_l \Pi_l$. This shows that either $t\pi < ((t_1 \Pi_1 \utimes t_2 \Pi_2) \utimes \dots) \utimes t_l \Pi_l$ or they share the same breaking points. In the latter case, $\{ \vert \Pi_1 \vert, \vert \Pi_1 \vert + \vert \Pi_2 \vert, \dots \}$ are breaking points of $t\pi$, and since $t \st(\pi\vert_{S_j}) = t_j \Pi_j$, we have $\pi = \Pi_1 \utimes \Pi_2 \utimes \dots \utimes \Pi_l$.
	
	Now, consider the factorization $t\pi = ((t'_1 \Pi'_1 \utimes t'_2 \Pi'_2) \utimes \dots ) \utimes t'_l \Pi'_l$ in connected compermutations:
	\begin{enumerate}
		\item If $\vert \Pi'_{l'} \vert > \vert \Pi_l \vert$, then $t\pi < ((t_1 \Pi_1 \utimes t_2 \Pi_2) \utimes \dots ) \utimes t_l \Pi_l$. 
		\item If $\vert \Pi'_{l'} \vert < \vert \Pi_l \vert$, then by restriction to $S_l$, $t_l \Pi_l$ would take the form $t'\Pi' \utimes t'_l \Pi'_{l'}$, which is impossible since both $t'_{l'} \Pi'_{l'}$ and $t_l \Pi_l$ are connected.
		\item If $\vert \Pi_l \vert = \vert \Pi'_{l'} \vert$, then by restriction to $S_l$, $t_l \Pi_l = t'_{l'} \Pi'_{l'}$. Iterating this shows that either $t\pi < ((t_1 \Pi_1 \utimes t_2 \Pi_2) \utimes \dots ) \utimes t_l \Pi_l$ or they are identical.
	\end{enumerate}
	Thus, the summands in $((t_1 \Pi_1 *' t_2 \Pi_2)*' \dots )*' t_l \Pi_l$ are either the compermutation $((t_1 \Pi_1 \utimes t_2 \Pi_2) \utimes \dots) \utimes t_l \Pi_l$ or compermutations strictly lower with respect to $<$. The statement follows.
\end{proof}

A graded coassociative coalgebra $C = \bigoplus_{n \geq 0} C_n$ is said to be cofree if for any graded coassociative coalgebra $D$ and any linear map $f \colon D \longrightarrow C_1$ (where $f(D_i) = 0$ for $i \neq 1$), there exists a unique morphism of graded coalgebras $\hat{f} \colon D \longrightarrow C$ such that $p_1 \hat{f} = f$, where $p_1$ is the canonical projection onto $C_1$. By dualizing Theorem~\ref{thm:structure_dual}, we obtain:

\begin{corollary}\label{cor:classification_Delta}
	$(\field[\comperm], \Delta, \epsilon)$ is cofree as a graded coassociative coalgebra.
\end{corollary}

\subsection{The structure of $(\field[\comperm_m],\circ)$}
In this part, we discuss the structure of the associative algebra $(\field[\comperm_m], \circ)$. The results presented here, and in particular Theorem~\ref{thm:structure_inner}, are closely related to general phenomena in the theory of monoid algebras and may be familiar to readers with experience in related combinatorial or geometric settings, such as face algebras. Since we were not able to locate a reference where this specific formulation is treated explicitly, we include both a proof based on general monoid‑theoretic arguments \cite{MR3525092} and a direct elementary proof.

Recall that, viewing binary rooted planar trees as a nonsymmetric magmatic operad, we have
\begin{displaymath}
	t\pi \circ t' \pi' = t'(t,\dots, t) \pi \circ \pi',
\end{displaymath}
where $\pi \circ \pi'$ consists of the subsets in the matrix
\begin{equation}\label{eq:main_matrix}
	\left(
	\begin{array}{ccc}
		\pi(\pi'_1) \cap \pi_1 & \dots &  \pi(\pi'_1)  \cap \pi_r  \\
		\vdots && \vdots \\
		\pi(\pi'_s) \cap \pi_1 & \dots &  \pi(\pi'_s)  \cap \pi_r  \\
	\end{array}
	\right)
\end{equation}
read row by row.

\begin{proposition}\label{prop:quotient_compermutations_to_compositions}
	The kernel of the homomorphism
	\begin{eqnarray*}
		\zeta \colon (\field[\comperm_m],\circ) &\longrightarrow& (\field[\comp_m],\circ)\\
		t\pi & \mapsto & \pi
	\end{eqnarray*}
	is a nilpotent ideal.
\end{proposition}

\begin{proof}
	$\ker \zeta$ is linearly spanned by $K = \{t \pi - t' \pi \mid t\pi, t'\pi \in \comperm_m\}$. If $\longi(\pi \circ \pi') \leq \longi(\pi')$, then $t\pi \circ t' \pi' = t' \pi \circ \pi'$ since each row of the matrix (\ref{eq:main_matrix}) contains only one non-empty entry. Thus, the product of $m$ elements of $K$ either vanishes or results in a linear combination of elements $t'\pi'$ with $\pi'$ having length $\geq m+1$, which is impossible.
\end{proof}

\begin{example} 
	Since the underlying permutation of $\pi = (3,12)$ is $\sigma^{\pi} = (3,1,2)$, then $(3,12) \circ (23,1)$ is obtained from
	\begin{align*}
		(\pi(\{2,3\}) &\cap \{3\}, \pi(\{2,3\}) \cap \{1,2\}, \pi(\{1\}) \cap \{3\}, \pi(\{1\}) \cap \{1,2\})\\
		& = (\{1,2\} \cap \{3\}, \{1,2\} \cap \{1,2\}, \{3\} \cap \{3\}, \{3\}  \cap \{1,2\})\\
		& = (\emptyset, \{1,2\}, \{3\},\emptyset)
	\end{align*}
	after removing empty components; i.e., $(3,12) \circ (23,1) = (12,3)$. Similarly, defining $\langle\pi\rangle := \pi - \sigma^\pi$ for any $\pi \in \comp_3$, we obtain the following multiplication table, which corresponds to the 2 by 2 matrix algebra:
	\begin{displaymath}
		\begin{array}{c|cccc}
			\circ	 & \langle (12,3)\rangle 	& \langle(1,23)\rangle 	& \langle(3,12)\rangle 	& \langle(23,1)\rangle 	\cr
			\hline
			\langle(12,3)\rangle & \langle(12,3)\rangle 	& 0 		& \langle(3,12)\rangle  & 0 		\cr
			\langle(1,23)\rangle & 0			& \langle(1,23)\rangle	& 0 		& \langle(23,1)\rangle	\cr
			\langle(3,12)\rangle & 0 			& \langle(3,12)\rangle	& 0			& \langle(12,3)\rangle	\cr
			\langle(23,1)\rangle & \langle(23,1)\rangle		& 0 		& \langle(1,23)\rangle	& 0
		\end{array}
	\end{displaymath}
\end{example}

We now determine the structure of $\field[\comp_m]$. An interval of $\conj{m}$ is a subset of the form $a, a+1, \dots, a+l$. We say that $\pi$ consists of intervals if all its components are intervals. Consider the set
\begin{displaymath}
	\Int_m:= \{\pi \in \comp_m \mid \pi \text{ consists of intervals} \},
\end{displaymath}
and the subspace $\field[\Int_m]$ spanned by $\Int_m$. For any $\pi_i \subseteq \conj{m}$, there is a unique partition into maximal intervals $(\pi_{i,1},\dots,\pi_{i,k})$. We define the map
\begin{eqnarray*}
	\eta \colon (\field[\comp_m],\circ) & \longrightarrow & (\field[\Int_m],\circ) \\
	\pi  & \mapsto &  (\pi_{1,1},\pi_{1,2},\dots, \pi_{2,1},\pi_{2,2},\dots).
\end{eqnarray*}
For a composition $I = (i_1,\dots,i_r)$ of $m$, let $\Stb(I)$ be its stabilizer in $\bS_r$ under the right action $I\sigma = (i_{\sigma(1)}, \dots, i_{\sigma(r)})$, and $\field[\Stb(I)]$ its group algebra. We will also recall right actions on pseudo-compositions $\pi\sigma = (\pi_{\sigma(1)}, \dots, \pi_{\sigma(r)})$.

\begin{theorem}\label{thm:structure_inner}
	The map $\eta$ is an algebra homomorphism, and $\ker\eta = \spann\langle \pi - \eta(\pi) \mid \pi \in \comp_m \rangle$ is a nilpotent ideal of $\field[\comp_m]$. Furthermore, if the characteristic of $\field$ does not divide $\vert \Stb(I) \vert$ for any partition $I$ of $m$, then $(\field[\Int_m], \circ)$ is semisimple, $\Nil(\field[\comp_m]) = \ker \eta$, and there exists a direct sum decomposition:
	\begin{displaymath}
		\field[\comp_m] = \field[\Int_m] \oplus \Nil(\field[\comp_m]).
	\end{displaymath}
	Moreover, the following isomorphism holds:
	\begin{displaymath}
		\field[\Int_m] \cong \bigoplus_{I} M_{\vert\bS_r: \Stb(I)\vert}(\field[\Stb(I)]),
	\end{displaymath}
	where $I=(i_1, \dots, i_r)$, with $i_1 \geq i_2 \geq \dots \geq i_r$, ranges over all partitions of $m$.
\end{theorem}

\begin{proof}
	We first verify that $\field[\Int_m]$ is a subalgebra of $\field[\comp_m]$. Let $\pi, \pi' \in \Int_m$. Given elements $a, b \in \pi(\pi'_i) \cap \pi_j$ such that $a < b$, it follows that the consecutive integers $a = \pi(a'), a+1 = \pi(a'+1), \dots, b = \pi(b')$ are contained in $\pi_j$. Since $\pi'_i$ is an interval, we must have $a', a'+1, \dots, b' \in \pi'_i$, which confirms that $\pi(\pi'_i) \cap \pi_j$ is indeed an interval.
	
	To show that $\eta$ is an algebra homomorphism, note that $\eta(\pi)$ is uniquely determined by the maximal intervals of $\pi$ serving as its components, while preserving the underlying permutation $\sigma^{\eta(\pi)} = \sigma^{\pi}$. Thus, proving $\eta(\pi \circ \pi') = \eta(\pi) \circ \eta(\pi')$ reduces to showing that the maximal intervals of $\pi \circ \pi'$ are precisely the intersections $\pi(S') \cap S$, where $S$ and $S'$ are maximal intervals of $\pi$ and $\pi'$, respectively. If a set $\{a_1, a_2, \dots\}$ is contained in $\pi(\pi'_i) \cap \pi_j$, there exist maximal intervals $S \subseteq \pi$ and $S' \subseteq \pi'$ such that $\{a_1, a_2, \dots\} \subseteq \pi(S') \cap S$. Since this intersection is an interval, the assertion follows.
	
	The nilpotency of $\ker \eta = \spann\langle \pi - \pi' \mid \eta(\pi) = \eta(\pi') \rangle$ is established by examining the matrix (\ref{eq:main_matrix}). We observe that $\longi(\pi \circ \pi') \leq \longi(\pi)$ if and only if the only non-empty entry in the $j$-th column is $\pi_j$. Since $(\sigma^{\pi})^{-1}(\pi_i)$ is an interval, the same holds for the matrix associated with $\pi \circ \eta(\pi')$; in fact, removing empty components yields $\pi \circ \pi' = \pi \circ \eta(\pi')$. Consequently, if $\longi(\pi \circ \pi') \leq \longi(\pi)$, then $\pi \circ (\pi' - \eta(\pi')) = 0$. This condition ensures that any product of at least $m$ factors in $\ker \eta$ must vanish.
	
	Once the semisimplicity of $\field[\Int_m]$ is established, the equality $\Nil(\field[\comp_m]) = \ker \eta$ and the stated direct sum decomposition follow immediately. The semisimplicity of $\field[\Int_m]$ is a consequence of its structure as a direct sum of matrix algebras, as detailed below. We note that the restriction on the characteristic of $\field$ is required to ensure that each group algebra $\field[\Stb(I)]$ is semisimple.
	
	We adopt the notation:
	\begin{displaymath}
		\b1_I= (\{1,2,\dots, i_1\}, \dots, \{i_1 + \dots + i_{r-1} + 1, \dots, m\}), 
	\end{displaymath}
	for $I = (i_1,\dots, i_r) \models m$. It follows from (\ref{eq:main_matrix}) that the subspaces
	\begin{displaymath}
		S_{\geq r} := \spann\langle \pi \in \Int_m \mid \longi(\pi) \geq r \}
	\end{displaymath}
	form a filtration of ideals of $S:=\field[\Int_m]$. Let $[\pi]$ denote the equivalence class $\pi + S_{\geq r+1}$.
	
	Now, consider the quotient algebras $S_{\geq r} / S_{\geq r+1}$. For each partition $I$ of $m$ with length $r$, define:
	\begin{displaymath}
		S_I / S_{\geq r+1} := \spann \langle [\b1_{I\sigma} \tau] \mid \sigma, \tau \in \bS_r \rangle.
	\end{displaymath}
	This space is spanned by classes $[\pi]$ where $\pi$ consists of intervals whose lengths lie in the orbit of $I$ under the $\bS_r$-action. It is clear that $S_I / S_{\geq r+1}$ is an ideal of $S_{\geq r} / S_{\geq r+1}$, leading to the decomposition:
	\begin{displaymath}
		S_{\geq r} / S_{\geq r+1} = \bigoplus_{I \vdash m, \longi(I) = r} S_I / S_{\geq r+1}.
	\end{displaymath}
	
	Finally, the algebra $(S_I / S_{\geq r+1}, \circ)$ can be identified with the tensor product $\spann \langle J \mid J \in \Orb(I) \rangle \otimes \field[\bS_r]$ equipped with the product:
	\begin{displaymath}
		(J \otimes \tau)(J' \otimes \tau') = \left\{ \begin{array}{ll} J \otimes \tau\tau' & \text{ if } J\tau = J' \\ 0 & \text{otherwise} \end{array}\right.
	\end{displaymath}
	By considering the left cosets $\sigma_k \Stb(I)$ of $\Stb(I)$ in $\bS_r$ and comparing dimensions, we obtain:
	\begin{displaymath}
	\spann \langle J \mid J \in \Orb(I) \rangle  \otimes \field[\bS_r] = \bigoplus_{i,j} I \sigma^{-1}_i \otimes \sigma_i \field[\Stb(I)] \sigma^{-1}_j.
	\end{displaymath}
	Since
	\begin{displaymath}
		(I \sigma^{-1}_i \otimes \sigma_i b \sigma_j^{-1})(I \sigma^{-1}_k \otimes \sigma_k b' \sigma_k^{-1}) = \delta_{j,k} I \sigma^{-1}_i \otimes \sigma_i bb' \sigma_l^{-1},
	\end{displaymath}
	we conclude that $$\spann \langle J \mid J \in \Orb(I) \rangle  \otimes \field[\bS_r] \cong M_{\vert \bS_r : \Stb(I)\vert}(\field[\Stb(I)]).$$
\end{proof}

\begin{example}
	The partitions of $m = 3$ are $(3)$, $(2,1)$, and $(1,1,1)$. It is straightforward to see that $\Stb_{\bS_1}((3)) = \{(1)\}$, $\Stb_{\bS_2}((2,1)) = \{(1,2)\}$, and $\Stb_{\bS_3}((1,1,1)) = \bS_3$, with indices $1$, $2$, and $1$, respectively. Consequently, we obtain the isomorphism:
	\begin{displaymath}
		\field[\comperm_3]/\Nil(\field[\comperm_3]) \cong \field \oplus M_2(\field) \oplus \field[\bS_3].
	\end{displaymath}
	The kernel of the projection $\eta$ is spanned by the set $\{(13,2) - (1,3,2), (2,13) - (2,1,3)\}$. One may readily verify that this is a nilpotent ideal:
	\begin{align*}
		((13,2) - (1,3,2)) \circ ((13,2) - (1,3,2)) &\\
		&\hskip -3cm = (1,2,3) - (1,2,3) - (1,2,3) + (1,2,3) = 0, \\
		((13,2) - (1,3,2)) \circ ((2,13) - (2,1,3)) &= 0, \\
		((2,13) - (2,1,3)) \circ ((13,2) - (1,3,2)) &= 0, \\
		((2,13) - (2,1,3)) \circ ((2,13) - (2,1,3)) &= 0.
	\end{align*}
	While the radical of $\field[\comp_3]$ has a nilpotency index of $2$ due to its low dimension, this property does not hold in general. For instance, in $\field[\comp_4]$ we observe:
	\begin{align*}
		((134,2) - (1,34,2)) \circ ((13,24) - (1,3,2,4)) = (14,3,2) - (1,4,3,2) \neq 0.
	\end{align*}
\end{example}

\subsubsection{A proof via the structure theory of monoid algebras}
In this part, we will freely use the notation and definitions in \cite{MR3525092}. The operation $\circ$ admits an alternative interpretation. Consider the symmetric group $\bS_m$ and the lattice $(\Partitions_m, \wedge)$ of partitions of $\conj{m}$, equipped with the natural $\bS_m$-action $P^{\tau} = \tau^{-1}(P)$. For $P = \{P_1, \dots, P_r\}$ and $Q = \{Q_1, \dots, Q_s\}$, the meet $P \wedge Q$ is defined as $\{ P_i \cap Q_j \mid i = 1, \dots, r, \, j = 1, \dots, s \}$, where empty subsets are discarded. We define the semidirect product $\bS_m \ltimes \Partitions_m$ with the multiplication:
\begin{displaymath}
	(\sigma P)(\tau Q) := \sigma\tau P^{\tau} \wedge Q.
\end{displaymath}
To each composition $\pi = (\pi_1, \dots, \pi_r)$, we associate two partitions: $P_\pi = \{ \pi_1, \dots, \pi_r\}$, obtained by neglecting the order of components, and $\b1_\pi$, obtained by partitioning $\conj{m}$ into $r$ consecutive intervals of lengths $|\pi_1|, |\pi_2|, \dots, |\pi_r|$. Specifically, $\b1_\pi = \Sh(\pi)(P_\pi)$. Under this notation, we have:
\begin{align*}
	(\sigma^\pi \b1_\pi)(\sigma^{\pi'} \b1_{\pi'}) &= \sigma^\pi \sigma^{\pi'} (\Sh(\pi')(\b1_\pi) \wedge \b1_{\pi'}) = \sigma^{\pi \circ \pi'}(\Sh(\pi') \Sh(\pi)(P_\pi) \wedge \b1_{\pi'}) \\
	&= \sigma^{\pi \circ \pi'}(\Sh(\pi')\Sh(\pi)(P_\pi \wedge P_{\pi(\pi')})) = \sigma^{\pi \circ \pi'} (\Sh(\pi') \Sh(\pi)(P_{\pi \circ \pi'})) \\
	&= \sigma^{\pi \circ \pi'} \b1_{\pi \circ \pi'}.
\end{align*}
Assuming as before that empty components in $\pi \circ \pi'$ are removed, it follows that $(\comp_m, \circ)$ is isomorphic to a submonoid of $\bS_m \ltimes \Partitions_m$. The latter is an inverse monoid where the star operation is given by $(\sigma P)^* := \sigma^{-1} \sigma(P)$, and the idempotents are of the form $(1, \dots, m) P$, which we identify with $P \in \Partitions_m$. The minimal ideal of $\bS_m \ltimes \Partitions_m$ is the set $\{\sigma \{\{1\}, \dots, \{m\}\} \mid \sigma \in \bS_m\}$, which constitutes a group isomorphic to $\bS_m$. 

\begin{lemma}
	Let $\comp'_m$ be the set of elements $\sigma P \in \bS_m \ltimes \Partitions_m$ such that $\sigma|_{P_i}$ is increasing and $P_i$ is an interval for each $P_i \in P$. Then $\comp'_m$ is a submonoid of $\bS_m \ltimes \Partitions_m$ isomorphic to $\comp_m$.
\end{lemma}

\begin{proof}
	We first demonstrate that $\comp'_m$ is closed under the monoid operation. Let $\sigma P, \tau Q \in \comp'_m$. Suppose $\tau^{-1}(a), \tau^{-1}(b) \in \tau^{-1}(P_i) \cap Q_j$ with $\tau^{-1}(a) < \tau^{-1}(a) + 1 \leq \tau^{-1}(b)$. Since $\tau|_{Q_j}$ is increasing, it follows that $a, b \in P_i$ and $a < \tau(\tau^{-1}(a) + 1) \leq b$. Given that $P_i$ and $Q_j$ are intervals, we observe that $\tau(\tau^{-1}(a) + 1) \in P_i$ and $\tau^{-1}(a) + 1 \in \tau^{-1}(P_i) \cap Q_j$, thereby proving that $\tau^{-1}(P_i) \cap Q_j$ is an interval. Furthermore, the monotonicity of $\sigma|_{P_i}$ implies $\sigma(a) = \sigma\tau(\tau^{-1}(a)) < \sigma\tau(\tau^{-1}(b))$. Consequently, $\sigma\tau$ is increasing when restricted to any non-empty subset in $P^{\tau} \wedge Q$, which confirms that $(\sigma P) (\tau Q) \in \comp'_m$.
	
	Finally, it is straightforward to verify that the map
	\begin{align}
		\comp_m & \longrightarrow \comp'_m \nonumber\\
		\pi & \mapsto \sigma^{\pi} (\sigma^{\pi})^{-1}(\{\pi_1, \dots, \pi_r\}) \label{eq:isomorphism_monoids_compositions}
	\end{align}
	establishes an isomorphism of monoids.
\end{proof}

Note that $\comp'_m$ is not necessarily closed under the operation $\sigma P \mapsto (\sigma P)^*$, as $\sigma(P)$ may not consist of intervals. It is therefore natural to consider its (von Neumann) regular elements. Since $\comp'_m$ is a monoid with commuting idempotents, the regular elements constitute a submonoid (cf. \cite{MR3525092}, p. 34, Exercise 3.4).

\begin{lemma}
	An element $\sigma P \in \comp'_m$ is regular if and only if $\sigma(P)$ consists of intervals.
\end{lemma}

\begin{proof}
	The element $\sigma P$ is regular if there exists $\tau Q$ such that $\sigma P = \sigma P \tau Q \sigma P$, which is equivalent to requiring $\tau = \sigma^{-1}$ and $\sigma(P) = \sigma(P) \wedge Q$. In this setting, let $\sigma(a) < \sigma(a) + 1 \leq \sigma(b) \in \sigma(P_i) = \sigma(P_i) \cap Q_j$. Since $\sigma^{-1}|_{Q_j}$ is increasing and $P_i$ is an interval, we have $a < \sigma^{-1}(\sigma(a) + 1) \leq b$. As $a$ and $b$ both belong to the interval $P_i$, it follows that $\sigma^{-1}(\sigma(a) + 1) \in P_i$, which implies $\sigma(a) + 1 \in \sigma(P_i)$. Thus, $\sigma(P)$ consists of intervals. Conversely, if $\sigma P \in \comp'_m$ and $\sigma(P)$ consists of intervals, then $\sigma^{-1}\sigma(P) \in \comp'_m$ and $\sigma P = \sigma P \sigma^{-1}\sigma(P) \sigma P$, proving that $\sigma P$ is regular.
\end{proof}

Let $\Int'_m$ denote the submonoid of regular elements in $\comp'_m$. Under the isomorphism \eqref{eq:isomorphism_monoids_compositions}, $\Int'_m$ is identified with $\Int_m$.

\begin{lemma}
	The following properties hold in $\Int'_m$:
	\begin{enumerate}
		\item Two idempotents $P, Q \in \Int'_m$ are $\mathscr{J}$-equivalent if and only if they have the same multiset of block sizes, that is,
		$\{|P_1|,\dots,|P_{\longi(P)}|\}
		=
		\{|Q_1|,\dots,|Q_{\longi(Q)}|\}$.
		\item If $P = \{P_1, \dots, P_r\}$, then $|E(J_{P})| = |\bS_r : \Stb(I)|$, where $I$ is the partition associated with $P$.
		\item $G_{P} \cong \Stb(I)$.
	\end{enumerate}
\end{lemma}

\begin{proof}
	Two idempotents $P$ and $Q$ are $\mathscr{J}$-equivalent if and only if they generate the same principal ideal. This is easily shown to be equivalent to $P = \sigma(Q)$ for some $\sigma \in \bS_m$, i.e., their component cardinalities coincide as multisets. Furthermore, $E(J_P)$ corresponds to the orbit of $P$ under the action of $\bS_r$, hence $|E(J_{P})| = |\bS_r : \Stb(I)|$. By \cite{MR3525092} (p. 28, Corollary 3.6), we have $G_P = \{ \sigma P \mid \sigma(P) = P \} \cong \Stb(I)$.
\end{proof}

The classification of $\field[\Int'_m]$ follows from Clifford-Munn-Ponizvski\u{\i}  theory \cite{MR3525092} which states that $\field M \cong \prod_{i=1}^s M_{n_i}(\field [G_{e_i}])$, where $e_1, \dots, e_s$ are idempotent representatives of the $\mathscr{J}$-classes of $M$ and $n_i = |E(J_{e_i})|$.

%
%
\section{Appendix: examples of primitive elements}

We will compute a basis for the space of primitives of $\field[\comperm_3]$ to show the interesting interplay of trees and compositions in $\field[\comperm]$. Clearly
\begin{displaymath} 
	\dim \field[\comperm_m] = \sum_{k = 0}^m k! S_{m,k}C_{k-1}
\end{displaymath}
where $S_{m,k}$ denotes the Stirling numbers of the second kind and $C_{k-1}$ is the Catalan number that counts the number of nonassociative monomials $w(x)$ of degree $k$. Thus, for $m = 0, 1, 2, 3$ we get  $\dim \field[\comperm_m] = 1, 1, 3, 19$. The number of connected compermutations in $\comperm_2$ (resp. $\comperm_3$) is $3-1 = 2$ (resp. $19 - 2\cdot 2 - 1 = 14$). Compermutations in $\comperm_3$ are given by the following table:
\begin{displaymath}
	\begin{array}{c|cccccc}
		\gbeg{1}{3}
		\linea  				\gnl
		\linea  				\gnl 
		\gend
		& (123) & & & & & \\ \hline  \\ 
		\gbeg{2}{3}
		\linea \linea 				\gnl
		\gmu            			\gnl
		\gend
		& (12,3) & (13,2) & (23,1) & (3,12) & (2,13) & (1,23)  \\ \hline \\ 
		\gbeg{3}{4}
		\gmu  					\linea                  		\gnl
		\gcn{2}{1}{2}{3}  		\linea 							\gnl
		\gvac{1} 				\gmu                  			\gnl
		\gend 
		& (1,2,3) & (1,3,2) & (2,1,3)  & (2,3,1) & (3,1,2)  & (3,2,1) \\ \hline \\ 
		\gbeg{3}{4}
		\linea					\gmu  					    	\gnl
		\linea \gcn{2}{1}{2}{1}  						\gnl
		\gmu                  			\gnl
		\gend
		& (1,2,3) & (1,3,2) & (2,1,3)  & (2,3,1) & (3,1,2)  & (3,2,1) 
	\end{array} 
\end{displaymath}
where, as usual in examples, for short, we write $(12,3)$ instead of $(\{1,2\},\{3\})$, etc. 
The non-connected compermutation in $\comperm_2$ is 
$
\gbeg{2}{2}
\gmu            			\gnl
\gend 
(1,2)
$, while the non-connected compermutations in $\comperm_3$ are
\begin{equation}\label{eq:non_connected}
	\gbeg{2}{3}
	\linea \linea 				\gnl
	\gmu            			\gnl
	\gend (12,3), \quad
	\gbeg{2}{3}
	\linea \linea 				\gnl
	\gmu            			\gnl
	\gend (1,23), \quad
	\gbeg{3}{4}
	\gmu  					\linea                  		\gnl
	\gcn{2}{1}{2}{3}  		\linea 							\gnl
	\gvac{1} 				\gmu                  			\gnl
	\gend (1,2,3), \quad
	\gbeg{3}{4}
	\gmu  					\linea                  		\gnl
	\gcn{2}{1}{2}{3}  		\linea 							\gnl
	\gvac{1} 				\gmu                  			\gnl
	\gend (2,1,3), \quad
	\gbeg{3}{4}
	\linea					\gmu  					    	\gnl
	\linea \gcn{2}{1}{2}{1}  						\gnl
	\gmu                  			\gnl
	\gend (1,3,2).
\end{equation}
For instance the last compermutation in \eqref{eq:non_connected} is of the form
$\gbeg{1}{3}
\linea  				\gnl
\linea  				\gnl 
\gend (1) \utimes \gbeg{2}{3}
\linea \linea 				\gnl
\gmu            			\gnl
\gend 
(2,1)
$ with both factors connected, so it is not connected. 

To compute a basis for the space of primitive elements of $\field[\comperm]$, we first consider the connected $t\pi$ with respect to the product $*'$, which form a basis of $\field[\comperm]$ and then we take the dual of this basis. In particular, the dimension of the space of primitive elements in $\field[\comperm_m]$ coincides with the number of connected compermutations in $\comperm_m$. After some computations we get the following basis for the space of primitive elements of $\field[\comperm_3]$:
\begin{enumerate}
	\item $
	\gbeg{1}{3}
	\linea  				\gnl
	\linea  				\gnl 
	\gend (123)
	- 
	\gbeg{2}{3}
	\linea \linea 				\gnl
	\gmu            			\gnl
	\gend (12,3)
	- 
	\gbeg{2}{3}
	\linea \linea 				\gnl
	\gmu            			\gnl
	\gend (1,23)
	+
	\gbeg{3}{4}
	\gmu  					\linea                  		\gnl
	\gcn{2}{1}{2}{3}  		\linea 							\gnl
	\gvac{1} 				\gmu                  			\gnl
	\gend (1,2,3)$
	\item $
	\gbeg{2}{3}
	\linea \linea 				\gnl
	\gmu            			\gnl
	\gend (13,2)
	-
	\gbeg{3}{4}
	\linea					\gmu  					    	\gnl
	\linea \gcn{2}{1}{2}{1}  						\gnl
	\gmu                  			\gnl
	\gend (1,3,2)$
	\item $
	\gbeg{2}{3}
	\linea \linea 				\gnl
	\gmu            			\gnl
	\gend  (23,1)
	-
	\gbeg{2}{3}
	\linea \linea 				\gnl
	\gmu            			\gnl
	\gend (1,23)
	+
	\gbeg{3}{4}
	\gmu  					\linea                  		\gnl
	\gcn{2}{1}{2}{3}  		\linea 							\gnl
	\gvac{1} 				\gmu                  			\gnl
	\gend (1,2,3)
	-
	\gbeg{3}{4}
	\gmu  					\linea                  		\gnl
	\gcn{2}{1}{2}{3}  		\linea 							\gnl
	\gvac{1} 				\gmu                  			\gnl
	\gend (2,1,3)$
	\item $
	-
	\gbeg{2}{3}
	\linea \linea 				\gnl
	\gmu            			\gnl
	\gend  (12,3)
	+
	\gbeg{2}{3}
	\linea \linea 				\gnl
	\gmu            			\gnl
	\gend (3,12)
	+
	\gbeg{3}{4}
	\gmu  					\linea                  		\gnl
	\gcn{2}{1}{2}{3}  		\linea 							\gnl
	\gvac{1} 				\gmu                  			\gnl
	\gend (1,2,3)
	-
	\gbeg{3}{4}
	\linea					\gmu  					    	\gnl
	\linea \gcn{2}{1}{2}{1}  						\gnl
	\gmu                  			\gnl
	\gend (1,3,2)$
	\item $
	\gbeg{2}{3}
	\linea \linea 				\gnl
	\gmu            			\gnl
	\gend (2,13)
	-
	\gbeg{3}{4}
	\gmu  					\linea                  		\gnl
	\gcn{2}{1}{2}{3}  		\linea 							\gnl
	\gvac{1} 				\gmu                  			\gnl
	\gend (2,1,3)$
	\item $
	\gbeg{3}{4}
	\gmu  					\linea                  		\gnl
	\gcn{2}{1}{2}{3}  		\linea 							\gnl
	\gvac{1} 				\gmu                  			\gnl
	\gend (1,3,2)
	-
	\gbeg{3}{4}
	\linea					\gmu  					    	\gnl
	\linea \gcn{2}{1}{2}{1}  						\gnl
	\gmu                  			\gnl
	\gend (1,3,2)$
	\item $
	-
	\gbeg{3}{4}
	\gmu  					\linea                  		\gnl
	\gcn{2}{1}{2}{3}  		\linea 							\gnl
	\gvac{1} 				\gmu                  			\gnl
	\gend (2,1,3)
	+ 
	\gbeg{3}{4}
	\gmu  					\linea                  		\gnl
	\gcn{2}{1}{2}{3}  		\linea 							\gnl
	\gvac{1} 				\gmu                  			\gnl
	\gend (2,3,1)$
	\item $
	\gbeg{3}{4}
	\gmu  					\linea                  		\gnl
	\gcn{2}{1}{2}{3}  		\linea 							\gnl
	\gvac{1} 				\gmu                  			\gnl
	\gend (3,1,2)
	-
	\gbeg{3}{4}
	\linea					\gmu  					    	\gnl
	\linea \gcn{2}{1}{2}{1}  						\gnl
	\gmu                  			\gnl
	\gend (1,3,2)$
	\item $
	\gbeg{3}{4}
	\gmu  					\linea                  		\gnl
	\gcn{2}{1}{2}{3}  		\linea 							\gnl
	\gvac{1} 				\gmu                  			\gnl
	\gend (1,2,3)
	-
	\gbeg{3}{4}
	\gmu  					\linea                  		\gnl
	\gcn{2}{1}{2}{3}  		\linea 							\gnl
	\gvac{1} 				\gmu                  			\gnl
	\gend (2,1,3)
	+
	\gbeg{3}{4}
	\gmu  					\linea                  		\gnl
	\gcn{2}{1}{2}{3}  		\linea 							\gnl
	\gvac{1} 				\gmu                  			\gnl
	\gend (3,2,1)
	-
	\gbeg{3}{4}
	\linea					\gmu  					    	\gnl
	\linea \gcn{2}{1}{2}{1}  						\gnl
	\gmu                  			\gnl
	\gend (1,3,2)$
	\item $
	-
	\gbeg{3}{4}
	\gmu  					\linea                  		\gnl
	\gcn{2}{1}{2}{3}  		\linea 							\gnl
	\gvac{1} 				\gmu                  			\gnl
	\gend (1,2,3)
	+
	\gbeg{3}{4}
	\linea					\gmu  					    	\gnl
	\linea \gcn{2}{1}{2}{1}  						\gnl
	\gmu                  			\gnl
	\gend (1,2,3)$
	\item $
	-
	\gbeg{3}{4}
	\gmu  					\linea                  		\gnl
	\gcn{2}{1}{2}{3}  		\linea 							\gnl
	\gvac{1} 				\gmu                  			\gnl
	\gend (2,1,3)
	+
	\gbeg{3}{4}
	\linea					\gmu  					    	\gnl
	\linea \gcn{2}{1}{2}{1}  						\gnl
	\gmu                  			\gnl
	\gend (2,1,3)$
	\item $
	-
	\gbeg{3}{4}
	\gmu  					\linea                  		\gnl
	\gcn{2}{1}{2}{3}  		\linea 							\gnl
	\gvac{1} 				\gmu                  			\gnl
	\gend (2,1,3)
	+
	\gbeg{3}{4}
	\linea					\gmu  					    	\gnl
	\linea \gcn{2}{1}{2}{1}  						\gnl
	\gmu                  			\gnl
	\gend (2,3,1)$
	\item $
	-
	\gbeg{3}{4}
	\linea					\gmu  					    	\gnl
	\linea \gcn{2}{1}{2}{1}  						\gnl
	\gmu                  			\gnl
	\gend (1,3,2)
	+
	\gbeg{3}{4}
	\linea					\gmu  					    	\gnl
	\linea \gcn{2}{1}{2}{1}  						\gnl
	\gmu                  			\gnl
	\gend (3,1,2)$	
	\item $
	\gbeg{3}{4}
	\gmu  					\linea                  		\gnl
	\gcn{2}{1}{2}{3}  		\linea 							\gnl
	\gvac{1} 				\gmu                  			\gnl
	\gend (1,2,3)
	-
	\gbeg{3}{4}
	\gmu  					\linea                  		\gnl
	\gcn{2}{1}{2}{3}  		\linea 							\gnl
	\gvac{1} 				\gmu                  			\gnl
	\gend (2,1,3)
	-
	\gbeg{3}{4}
	\linea					\gmu  					    	\gnl
	\linea \gcn{2}{1}{2}{1}  						\gnl
	\gmu                  			\gnl
	\gend (1,3,2)
	+
	\gbeg{3}{4}
	\linea					\gmu  					    	\gnl
	\linea \gcn{2}{1}{2}{1}  						\gnl
	\gmu                  			\gnl
	\gend (3,2,1)$
\end{enumerate}

%
%

%
\end{document}